\documentclass[a4paper, 12pt]{amsart}

\usepackage[utf8]{inputenc}
\usepackage{amsmath,amsfonts,amssymb,amsopn,amscd,amsthm}
\usepackage{comment}
\usepackage{dsfont}
\usepackage{graphicx}
\usepackage{color}
\usepackage[colorlinks]{hyperref}
\usepackage{epigraph}
\usepackage{todonotes}
\usepackage[left=3.5cm,top=2.5cm,bottom=2cm,right=3cm]{geometry}
\usepackage{tikz-cd}
\usepackage{caption}

\title[ Kyle Model ]
      { A unified approach to informed trading via Monge-Kantorovich duality }

\date{15th of October 2021}

\allowdisplaybreaks[4]

\DeclareRobustCommand{\SkipTocEntry}[5]{}

\DeclareMathOperator{\Hc}{ \textrm{ch} }

\DeclareMathOperator{\Var}{Var}
\DeclareMathOperator{\Cov}{Cov}

\def\1{{\mathbf 1}}

\def\pa{\partial}

\def\R{{\mathbb R}}

\def\P{{\mathbb P}}
\def\E{{\mathbb E}}

\def\Ac{{\mathcal A}}

\def\Fc{{\mathcal F}}

\def\Hc{{\mathcal H}}

\def\Sc{{\mathcal S}}

\def\Wc{{\mathcal W}}

\newtheorem{thm}{Theorem}[section]
\newtheorem{proposition}[thm]{Proposition}

\newtheorem{definition}[thm]{Definition}

\newtheorem{lemma}[thm]{Lemma}

\newtheorem{rmk}[thm]{Remark}
\newtheorem{assumption}[thm]{Assumption}

\numberwithin{equation}{section}
\numberwithin{figure}{section}
\date{\today}

\author{Reda Chhaibi}
\address{Université Paul Sabatier, 
Institut de Mathématiques de Toulouse,
118 route de Narbonne, F-31062 Toulouse Cedex 9
}
\email{reda.chhaibi@math.univ-toulouse.fr}

\author{Ibrahim Ekren}
\address{Florida State University, Department of
    Mathematics, 1017 Academic Way, Tallahassee, FL 32306}
\thanks{I. Ekren is supported in part by NSF Grant DMS 2007826.}
\email{iekren@fsu.edu}

\author{Eunjung Noh}
\address{Florida State University, Department of
    Mathematics, 1017 Academic Way, Tallahassee, FL 32306}
\email{enoh@fsu.edu}
\author{Lu Vy}
\address{Florida State University, Department of
    Mathematics, 1017 Academic Way, Tallahassee, FL 32306}
\email{ldv20@fsu.edu}

\begin{document}

\begin{abstract}
We solve a generalized Kyle model type problem using Monge-Kantorovich duality and backward stochastic partial differential equations. 

First, we show that the the generalized Kyle model with dynamic information can be recast into a terminal optimization problem with distributional constraints. Therefore, the theory of optimal transport between spaces of unequal dimension comes as a natural tool. 

Second, the pricing rule of the market maker and an optimality criterion for the problem of the informed trader are established using the Kantorovich potentials and transport maps. 

Finally, we completely characterize the optimal strategies by analyzing the filtering problem from the market maker's point of view. In this context, the Kushner-Zakai filtering SPDE yields to an interesting backward stochastic partial differential equation whose measure-valued terminal condition comes from the optimal coupling of measures.
\end{abstract}

\maketitle

\hrule
\tableofcontents
\hrule
\newpage


\section{Introduction}

The \textit{Kyle model}, introduced in Kyle's seminal paper \cite{kyle}, and its non-Gaussian extension in \cite{back1992} have been the canonical models in market microstructure theory for the analysis of strategic trading in the presence of private information. By establishing a link between information asymmetry and market liquidity, these papers answer fundamental questions of price dynamics and agents' behavior. In this paper, we provide a unified framework to study a generalized Kyle-Back model which is motivated by activist trading as in \cite{bal,collin2015shareholder}, Kyle-Back model with dynamic information as in \cite{back1998long,ccd}, and the problem of risk-averse informed trading with imperfect information which is developed in our companion paper \cite{cen}. Our methodology, based on optimal transport theory with general cost functions and backward stochastic partial differential equations, provides a unified framework to establish the existence of equilibrium in all known versions (to the authors) of continuous-time risk-neutral Kyle model\footnote{with the only exception of the unpublished note \cite{back2013liquidity}}. We also extend these existence results to any number of assets with general distributions and obtain the strategies of both agents as sensitivities of relevant value functions. In our companion paper \cite{cen}, our methodology leads to the existence of equilibrium in a variant of the Kyle-Back model with risk-averse informed trading with imperfect information by studying the forward-backward version of the backward stochastic partial differential equation studied in this paper.

In Kyle's original model, there are three types of market participants: non-strategic noise traders, the strategic risk-neutral informed trader who has private information on the normally-distributed terminal value of the asset, and a market maker. 
The market maker receives orders from both the informed trader and noise traders, but cannot distinguish them. So, the objective of market maker is to set the pricing rule by filtering the information of the informed trader. 
On the other hand, the informed trader's goal is to maximize expected terminal wealth by taking advantage of her superior information given the price set by the market maker. In this framework, the strategic behavior of the informed trader shows how fast private information is integrated to the price and the pricing rule of the market maker shows how the price reacts to the total demand. The dynamic version of the model reflects the reality that, in most markets, large investors split their orders into small pieces to minimize price impacts. The continuous-time versions in  \cite{kyle,back1992} are especially tractable since they show that the pricing rule of the market maker can be obtained as the solution of the heat equations whose final condition can be explicitly written using the cumulative distribution function of the price and the noise trading at maturity.

An impressive number of extensions of this Kyle-Back model have been studied in the literature; see, e.g.,  \cite{back1992, cho,  ccd, cd1,back2013liquidity, bal, back2020optimal,barger2021insider,cdf,choi2022trading,ekren2022kyle,ying2020pre,cd1,back2000imperfect,caldentey2010insider,baruch2002insider,las2007}. In \cite{back1998long,ccd}, which we call the dynamic information model, it is shown that the equilibrium can be achieved if the informed trader learns the fundamental price of the risky asset dynamically on $[0,T]$ instead of learning it at time $0$. In \cite{cho,baruch2002insider,bose2020kyle,bose2021multidimensional}, an equilibrium is established if the informed trader is risk-averse. Relations between liquidity, activism, and firm value are studied in \cite{bal}, by considering a generalized Kyle-Back model where the informed trader is a potential activist who can affect firm's liquidation value by paying cost of effort. \cite{bal} assumes that private information of the informed trader consists only of her own block size, not including the value of stock at terminal time. 
If the informed trader decides to become `active' at maturity, then the payoff from informed trader's liquidation at maturity could be a non-linear function of the number of shares held and the firm value. 

In \cite{back2020optimal}, it was shown that several important quantities in Kyle-Back model can be described using theory of optimal transport. The connection between optimal transport theory and the Kyle-Back model relies on the \textit{inconspicuous trading property} of the equilibrium, which means that the informed trader's strategy remain undetected to the market maker. Such a trading strategy imposes a distributional constraint on the total order flow that market maker receives at terminal time. In \cite{back2020optimal}, due to the fact that the payoff of the informed trader at maturity is linear in the price, the classical Kyle-Back model leads to an optimal transport problem with quadratic cost and the unconditional profit of the informed trader is related to the Wasserstein-2 distance between the distribution of the price and the distribution of the noise trading at maturity (which is also the distribution of total demand due to inconspicuousness).

In our model, we assume a generalized form of the terminal wealth of the informed trader. 
More specifically, in the classical Kyle-Back model, the terminal wealth of the informed trader is, by integration by parts, 
\begin{align*}
	(\beta + X_T) f(S_T) - \int_0^T H(r, X_r + Z_r) dX_r \ , 
\end{align*}
where $\beta$ is the random initial endowment of the informed trader, $X$ is the trading strategy of the informed trader, $Z$ is the total demand of noise traders, $f(S_T)$ is the value of the stock at maturity $T$, $S_T$ is the private information of the informed trader, and $P_t = H(t, X_t + Z_t)$ is the pricing rule. In particular, the first term describes profits of the informed trader from liquidation at maturity $T$.
In our work, we fix a function $V: \mathbb{R}^2 \rightarrow \mathbb{ R}$, and generalize the terminal wealth of the informed trader by considering the terminal wealth
\begin{align*}
	\Wc_T(\beta,S_T, X,H):=  V(\beta + X_T, S_T) - \int_0^T H(r, X_r + Z_r) dX_r \ .
\end{align*}
In \cite{bal, back2013liquidity,collin2015shareholder} the nonlinearity in $V$ is interpreted as activist trading. Loosely speaking, in \cite{bal, back2013liquidity,collin2015shareholder}, if the total position at maturity $\beta + X_T$ of the informed trader is large enough, the informed trader can take an active role in the governance of the company and potentially generate nonlinear returns from her position. Note that thanks to the generality of function $V$, to the best of our knowledge, our model covers all examples in the literature of continuous-time Kyle model with risk neutral agents. For example, classical Kyle-Back model corresponds to the case $V(x,s) = x s$.

Our methodology shows that existence of equilibrium results available in the literature can be reduced to the study of two problems. The first problem is an optimal transport problem where the surplus function is given by the function $V$. This transport problem can be exhibited by assuming that there exists an equilibrium where the strategy of the informed trader satisfies the inconspicuousness property of \cite{cho}. This property means that the distribution of the total demand $Y_T = X_T + Z_T$ is given. Note also that, since $(Z_t-\beta,S_t)$ is not controlled as such the joint distribution of $(Z_T-\beta,S_T)$ is also given in our framework. However, the distribution of the family $(Z_T-\beta,S_T, Y_T)$ is to be determined by the equilibrium condition. 
We show that the this joint distribution is so that $(Z_T-\beta,S_T,Y_T)$ achieves the optimal coupling of the optimal transport problem associated to $V$. We also show that the pricing rule of the market maker and the conditional expected profit of the informed trader can be obtained using the Kantorovich potentials of the optimal transport problem. 
The optimal transport problem also provides a simple optimality criterion for the control problem of the market maker. This criterion is to force the total demand to achieve $Y_T=I(Z_T-\beta,S_T)$, where $I$ is the optimal transport map that transports the law of $(Z_T-\beta,S_T)$ to the given terminal law of $Y_T$. 
We show that an equilibrium exists if the informed trader can enforce the equality $Y_T=I(Z_T-\beta,S_T)$ using inconspicuous strategies. In the literature, the existence of such strategies are established via either the black-box Doob's h-transform \cite[IV.39-40]{rogers2000diffusions} or Markov bridges \cite{back1992,ccd,back2020optimal,follmer1999canonical}. We show that all these cases reduce to the study of an ill-posed linear backward stochastic partial differential equation. This equation is a backward version of the Kushner's equation with terminal condition corresponding to the disintegration of the optimal coupling of $(Z_T-\beta,S_T,Y_T)$ along $Y_T$. Thanks to the inconspicuousness property, if this backward stochastic differential equation admits a solution, this solution has the same initial condition as the Kushner's equation. Thus, our methodology constructs the solution to the (forward) Kushner's equation using the backward stochastic partial differential equations whose final condition is given by the optimal transport problem. Additionally, we show that the strategy of the informed trader can be written as the sensitivity in the innovation process of the backward stochastic partial differential equations (the process which is traditionally called $Z$ process in the BSDE literature \cite{el1997backward}). Thus, we show that strategies of both agents can be interpreted as sensitivities of appropriate processes in the innovation process. 

Given the generality of our framework, we make two implicit assumptions to have the existence of equilibrium. The first assumption (which is simpler to check) is on the Monge-Kantorovich duality and the existence of both dual and primal optimizers for an optimal transport problem associated to the Kyle model we study. In examples we study in the paper, general results such as \cite{b,santambrogio2015optimal,villani} allows us to check this assumption. The second (and most restrictive) assumption is the solvability of a generally ill-posed backward stochastic partial differential equation. We show that in many cases of interest, the solution of this equation can be explicitly obtained and we can establish an equilibrium for the generalized Kyle model.

The rest of the paper is organized as follows. In Section \ref{s.statement} we introduce the generalized Kyle model and mention the examples of interest. In Section \ref{s.solution}, we show that the concept of inconspicuousness and optimal transport theory allows us to pinpoint a candidate equilibrium pricing rule and an optimality condition for the strategy of the informed trader. In Section \ref{s.bspde}, we study the filtering problem of the market maker as both a forward and backward equations to exhibit a candidate equilibrium strategy for the informed trader. Section \ref{s.examples} contains the examples we treat, and a sobering counter-example to the framework. 

\section{Statement of the problem}\label{s.statement}

We fix $T>0$, the maturity of the problem, $n\geq 1$ the number of risky assets and $\sigma$ a symmetric positive definite matrix of dimension $n$.
Let $(\Omega, \Fc,\P)$ be a probability space endowed with two $n$-dimensional independent Brownian motions $(Z,W)$ with $Z$ (resp. $W$) being a $\sigma^2$ (resp. $I_n$)-Brownian motion and two $\R^n$-valued random variables $S_0,\beta$ independent of $(Z,W)$. 

In the market we consider, there are $n$ risky assets whose prices $P_t\in \R^n$ at $t\in [0,T]$ will be determined by an equilibrium condition and a risk-free asset whose interest rate is assumed to be $0$. The prices of the risky assets will be announced to be $f(S_T)\in \R^n$ at time $T$,
where $(S_t)_{0\leq t\leq T}$ is the ($n$-dimensional) private information observed by the informed trader and $f$ a continuous function. As in \cite{back1998long,ccd}, we assume that 
\begin{align}\label{def:S}
    S_t=S_0+\int_0^t \sigma_r dW_r
\end{align}
for a 
symmetric semi-definite matrix valued $t\mapsto\sigma_t$ with $\int_0^T |\sigma_r|^2 dr<\infty$.

Similar to the classical Kyle's model in \cite{kyle}, three market participants interact during the time interval $[0,T]$. 
The first market participant is the informed trader who learns at time $t=0$ the value of $S_0\in \R^n$ and observes $S_t$ at time $t\in[0,T]$. Additionally, $\beta\in \R^n$ is her initial endowment in the risky assets. Similarly to the Kyle model with dynamic information \cite{back1998long,ccd}, in our framework, $(S_t,\beta)$ are the private information of the informed trader. 
We assume that $\Fc=(\Fc_t)_{t\in [0,T]}$ is the augmentation of the filtration generated by the Markov process $(S_t,\beta,Z_t)_{t\in [0,T]}$ so that $(S_0,\beta)$ is $\Fc_0$-measurable. In all equilibria we construct, the informed trader is able to infer $Z_t$ from the value of $P_t$. Thus, the sigma-algebra $\Fc_t$ represents the information of the informed trader at time $t$. 
The informed trader chooses a trading strategy $X$ which is an $n$-dimensional (continuous) $\Fc$-semimartingale with $X_0=0$. We provide below an admissibility condition for the set of trading strategies we consider. By definition, $X_t\in \R^n$ represents the number of shares purchased on $[0,t]$ in each risky asset by the informed trader. Since $\beta$ is the random initial endowment in the risky assets of the informed trader, $X_t + \beta$ is her total position in the risky assets at time $t\in [0,T]$. 

The second market participant is a (group of) noise trader(s) whose aggregate demand in risky assets is $Z_t\in \R^n$. Since $Z$ is a $\sigma^2-$Brownian motion, we assume their trades have constant covariance matrix $\sigma^2$ per unit of time. Our assumption is done for notational simplicity and can easily be relaxed.

The third market participant is the market maker who only observes the total order imbalance,
$Y = X+Z \ $ and quotes a price for the assets given his information according to the equilibrium condition that we provide below.  
The information of the market maker is the augmented filtration generated by $Y$ which is denoted $\Fc^Y$. Naturally $\Fc^Y \subset \Fc$. 

For all $t$, we denote $\nu_t=N(0,\sigma^2 t)$ the Gaussian distribution of $Z_t$. Given $\beta$, we also define $\tilde Z_t:=Z_t-\beta$ and denote $\mu_t$ the joint distribution of $(\tilde Z_t,S_t)$ (on $\R^{2n}$). Note that given the equality \eqref{def:S} and the fact that $\tilde Z_t$ is a Brownian motion starting at $-\beta$, this distribution can be directly computed from the joint law $\mu_0$ of $(\tilde Z_0,S_0)=(-\beta,S_0)$. We also denote $\mu^1_t$ (resp. $\mu^2_t$) the distribution of $\tilde Z_t$ (resp. $S_t$).

We now define the set of pricing rules of the market maker. 
\begin{definition}\label{def:pricing}
 A continuous function $H:[0,T]\times \R^n\mapsto \R^n$ is called a pricing rule if it satisfies the following conditions. 
 \begin{itemize}
 \item[(i)] $H$ is once in time and twice in space continuously differentiable on $(0,T)$.
 \item[(ii)] $H$ satisfies the following integrability condition
 $$\E\left[|H|^2(T,Z_T)+\int_0^T |H|^2(r,Z_r)dr\right]<\infty \ .$$
 \end{itemize}
 We denote by $\Hc$ the set of pricing rules. 
\end{definition}
 
Given the definition of pricing rules we can now define the set of admissible trading strategies of the informed trader. 
\begin{definition}\label{def:ad}
For a given pricing rule $H\in \Hc$, a continuous $\Fc$-semimartingale $X$ with $X_0=0$ is said to be an admissible trading strategy if 
\begin{align}\label{eq:ad}
   \int_0^T \E[|H|^2(r,X_r+Z_r)]dr<\infty \ .
\end{align}
The set of admissible trading strategies given $H$ is denoted $\Ac(H)$.
\end{definition}

In the classical Kyle's model, the profit of the informed trader from trading on $[0,T)$ is $\int_0^T(\beta+X_r)^\top dP_r$
where $P_r=H(r,Y_r)$ and $Y_r=X_r+Z_r$. Additionally, at time $T$ the informed trader obtains a profit 
$$(f(S_T)-P_{T})^\top(\beta+X_T)$$
from any potential mispricing at maturity. Thus, by an integration by parts formula as in \cite{back1992,cetin2021pricing,back2020optimal}, her realized terminal wealth is 
\begin{align*}
&(f(S_T)-P_{T})^\top(\beta+X_T)+\int_0^T(\beta+X_r)^\top dP_r\\
&=f(S_T)^\top(\beta+X_T) - \int_0^T  H(r,X_r+Z_r)^\top dX_r- \langle X,P\rangle_T \ ,
\end{align*}
where $d\langle X,P\rangle_t=\sum_{i=1}^n d\langle X^i,P^i\rangle_t$.
In the setting of activism trading of \cite{bal,back2013liquidity}, the informed trader can leverage her position and private information to obtain higher and potentially non-linear payoffs from her liquidation at maturity. Thus, in this paper, we generalize the classical Kyle's model to allow that for a given continuous function $V:\R^{2n}\mapsto \R$, the informed trader's realized wealth for a given strategy $X$ is
\begin{align}
\label{eq:wealth}
\Wc_T(\beta,S_T, X,H) := & \ {V(\beta+X_T,S_T)}\\
&- \int_0^T  H(r,X_r+Z_r)^\top dX_r- \langle X,H(\cdot,X_\cdot+Z_\cdot)\rangle_T \ . \notag
\end{align}
\begin{rmk}
As it is the case in the classical Kyle's model, $\beta$ is not needed if $V$ is linear in its first argument. 
\end{rmk}

We now define the equilibrium condition of the generalized Kyle's model which is based on \cite{bal,back2013liquidity}. 
\begin{definition}
A pair $(H^*,X^*)$ with $H^*\in \Hc$ and $X^*\in \Ac(H)$ forms an equilibrium if the following conditions hold. 
\begin{itemize}
    \item [(i)] If the market maker uses the pricing rule 
    $P_r=H(r,X_r+Z_r)$, then $X^*$  maximizes the expected wealth of the informed trader
    \begin{align}
\label{eq:problem2}
\sup_{\tilde X\in \Ac(H)}\E\left[ \Wc_T(\beta,S_T,\tilde X,H) \ | \ \Fc_0 \right].
\end{align}
\item [(ii)] If the informed trader uses the trading strategy $X^*$, then $H^*$ is rational in the sense that 
\begin{align}
\label{eq:mm1}
 H^*(t,X^*_t+Z_t) =\E[\pa_x V(Y^*_T-\tilde Z_T,S_T)|\Fc^{X^*+Z}_t] \ .
\end{align}
\end{itemize}
\end{definition}
\begin{rmk}
    The rationality condition \eqref{eq:mm1} means the quoted price is the expectation of terminal utility indifference price of the informed trader. 
    
Suppose that (instead of quoting the price) the market maker observes the price process $H^*(t,Y^*_t)$ and decides to fulfill only $Y_t$ shares of the total demand $Y^*_t$. Then, his position is $-Y_t$ and his realized profits from trading (against the price $H^*(t,Y^*_t)$) is 
$$-\int_0^T Y_r dH^*(r,Y^*_r)-Y_T\big(\pa_x V(Y^*_T-\tilde Z_T, S_T)-H^*(T,Y^*_T)\big) \ , $$
where $\pa_x V(Y^*_T-\tilde Z_T,S_T)$ is the price at which the informed trader agrees to trade at final time and $H^*(T,Y^*_T)$ is the quoted price at final time.

The condition \eqref{eq:mm1} means that for any bounded process $Y$ which is $\Fc^{Y^*}$-adapted, this realized profit has $0$ expectation 
$$\E\left[-\int_0^T Y_r dH^*(r,Y^*_r)-Y_T\big(\pa_x V(Y^*_T-\tilde Z_T, S_T)-H^*(T,Y^*_T)\big)\right]=0 \ .$$
Thus, the condition \eqref{eq:mm1} pinpoints $(H^*,Y^*)$ so that a representative risk-neutral market maker observing the realization of demand $Y^*$ is indifferent to fulfilling this total demand $Y^*$. 
\end{rmk}
\subsection{Examples}
We list here some examples of Kyle's model we are able to treat. In Section \ref{s.examples}, we provide the necessary assumptions and computations to establish the existence of equilibrium in these examples.
\subsubsection{Classical Kyle model with static information}\label{sss:kyle1}
If $\sigma_t=0$ in (\ref{def:S}), then the information on the terminal-time stock price is $S_0$ for all $t\in[0,T]$. Additionally, with the choice of $V(x,s) = xs$, the problem we study yields the classical Kyle-Back model with static information, see \cite{kyle,back1992,back2020optimal}.


\subsubsection{Classical Kyle model with dynamic information}\label{sss:kyle2} The choice $V(x,s) = xs$ leads to the model studied in \cite{back1998long}.

A non-Gaussian generalization of \cite{back1998long} appears in \cite{ccd} for $n=1$, which considers signals of the form
\begin{align*}
    S_t &= S_0 + \int_0^t a(r) b\!\left({\tau(r), S_r}\right) {\rm d} W_r 
\end{align*}
where
\begin{align*}
    \tau(t) &= c + \int_0^t a^2(r) \, {\rm d}r, \qquad \tau(1) = 1.
\end{align*}
If we define
$$
B(t,\xi) = \frac{1}{2}\int_0^1 \partial_{\xi}b (r,0) \, {\rm d}r + \int_0^{\xi} \frac{1}{b(t,y)} \, {\rm d}y
$$
and assume $b$ satisfies
$$
\partial_{t} b(t,s) + \frac{1}{2}b^2 (t,s) \partial_{ss}b(t,s) = 0
$$
as \cite{ccd} does, then it is easy to show that $\tilde{S}_t = B\!\left(\tau(t), S_t\right)$ satisfies
$$
\tilde{S}_t = \tilde{S}_0 + \int_0^t a(r) \, {\rm d}W_r.
$$
Thus, up to a deterministic transformation of the informed trader's private signal, \cite{ccd} satisfies our assumption \eqref{def:S}. If one also redefines $V(x,s) = xB^{-1}(1,s)$ (where the inverse is with respect to the second argument), our paper recovers the results in \cite{ccd}.

\subsubsection{Activist trading}\label{sss:kyle3}
Activist trading as in \cite{bal,back2013liquidity} corresponds to particular choices of $V$. For example, $V(x,s)=V(x)$ recovers an example from activism trading as in \cite{bal}. 


\subsubsection{Linear quadratic $V$}\label{sss:kyle5}
If we rewrite the equation (10) of \cite{collin2015shareholder}, assuming $X_0 = 0$, in the form of \eqref{eq:wealth}, one can see that it corresponds to the choice of  $$V(x,s) = \frac{\psi}{2}x^2 + xs$$ for some $\psi >0$ and $\sigma_t = 0$ in \eqref{def:S}.

\section{Market Maker's problem}\label{s.solution}
\subsection{Inconspicuousness}

Informally, inconspicuousness (as defined in \cite{cho}) is the idea that the informed trader's strategy $( X_t  )$ needs to remain undetected to the market maker, who is observing the process $Y = X + Z$. Our objective is to find a general methodology to generate equilibria where the inconspicuousness of the trading strategy of the informed trader holds. The following proposition serves as a formal definition of inconspicuousness.

\begin{proposition}\label{prop:inco}
A trading strategy $X$ is said to be inconspicuous if any of the following equivalent statements hold:
\begin{itemize}
\item The law of the process $Y = X + Z$ in its own filtration is the same as $Z$ and $d\langle X^i,Z^i\rangle_t\geq0$ for all $i=1,\ldots,n$ and $t\in[0,T]$.
\item The process $X$ has finite variation with $dX_t = A_t \ dt$ and 
$$ \E\left[ A_t \ | \ \Fc^Y_t \right] = 0 \ . $$
\end{itemize}
\end{proposition}
\begin{rmk}
Note that in both cases we easily infer $d\langle X^i,Z^i\rangle_t=0$ for all $i=1,\ldots,n$ and $t\in[0,T]$.
\end{rmk}
\begin{proof}[Proof of Proposition \ref{prop:inco}]
Assume that $Y = X + Z$ and $Z$ have the same law and $d\langle X^i,Z^i\rangle_t\geq0$ for all $i=1,\ldots,n$ and $t\in[0,T]$. Necessarily, quadratic variations are the same: 
$$d\langle Z^i\rangle_t=d\langle Y^i\rangle_t=d\langle X^i\rangle_t+2d\langle X^i,Z^i\rangle_t+d\langle Z^i\rangle_t \ .$$
Thus, by the condition $d\langle X^i,Z^i\rangle_t\geq 0$, we obtain that the quadratic variation of $X$ vanishes identically. Hence $dX_t = A_t \ dt$. Moreover, for every bounded $\Fc^Y$-predictable process $f$, we have:
$$
   \E\left( \int_0^T f_r dY_r \right)
   =
   \E\left( \int_0^T f_r dZ_r \right) \ ,  
$$
equivalently,
$$
   \E\left( \int_0^T f_r A_r dr  \right)
   =
   0 \ ,
$$
equivalently,
$$
   \E\left( \int_0^T f_r \ \E\left( A_r \ | \ \Fc^Y_r \right) dr  \right)
   =
   0 \ .
$$
Hence $\E\left( A_r \ | \ \Fc^Y_r \right) = 0$ for almost every $r$.

Now examining the reverse implication, we only need to prove that $Y=X+Z$ is a $\sigma^2$-Brownian motion given that $Z$ also is. Given that $X$ has finite variation, $Y$ has the correct quadratic variation. We thus only need to prove that $Y$ is an $\Fc^Y$- martingale.
\begin{align*}
   & \ \E\left( Y_{t+s}  - Y_t \ | \ \Fc^Y_t \right)\\
 = & \ \E\left( \int_{t}^{t+s} A_u du + Z_{t+s} - Z_t \ | \ \Fc^Y_t \right)\\
 = & \ \int_{t}^{t+s} \E\left( A_u \ | \ \Fc^Y_t \right) du \\
 = & \ \int_{t}^{t+s} \E\left( \ \E\left( A_u \ | \ \Fc^Y_u \right) \ | \ \Fc^Y_t \right) du\\
 = & \ 0 \ .
\end{align*}

\end{proof}

\subsection{Optimal transport}
By Proposition \ref{prop:inco}, for any inconspicuous trading strategy of the informed trader, the distribution of $Y_T$ is the same as that of $Z_T$ which is the Gaussian distribution $\nu_T.$
Thus, in the expression of the terminal wealth \eqref{eq:wealth}, the term $V(\beta+X_T,S_T)$ can be decomposed as 
$$V(\beta+X_T,S_T)=V(Y_T-\tilde Z_T,S_T) \ ,$$
where we separate the controlled process $Y_T$ from the uncontrolled random variables $(\tilde Z_T,S_T)=(\tilde Z_T-\beta,S_T)$.
Note that if we assume that the informed trader uses an inconspicuous strategy, in this expression the joint distribution of $(\tilde Z_T,S_T)$ is given to be $ \mu_T$ and the distribution of $Y_T$ is $\nu_T$. However, the joint distribution of the triplet
$(\tilde Z_T,S_T,Y_T)$ in equilibrium is not given and will be determined by an optimal transport problem.
Given $V$, we define the \textit{surplus function} as $$\Sc(\tilde z,s,y):={V(y-\tilde z,s)} \ .$$

Our main contribution is to show that there exists an equilibrium satisfying the inconspicuousness property and the joint distribution of $(\tilde Z_T,S_T,Y_T)$ is characterized by the optimal transport problem
\begin{align}
\label{eq:OT}
{\sup_{\pi \in \Pi( \mu_T  , \nu_T)} }\E^\pi[\Sc(\tilde Z_T,S_T,Y_T)] \ ,
\end{align}
where $\Pi( \mu_T , \nu_T)$ is the set of joint couplings $\pi$ of $(\tilde Z_T,S_T,Y_T)$ with the law of $Y_T$ being $\nu_T=N(0,\sigma^2 T)$ and the law of $( \tilde Z_T,S_T)$ being $\mu_T $.

We make the following implicit assumptions on $(\Sc, \mu_T,\nu_T)$. We show in Section \ref{s.examples} that these assumptions are satisfied in all Kyle models considered in the literature\footnote{with the only exception of the unpublished note \cite{back2013liquidity}.}. 

\begin{assumption}\label{ass.mk}
$\Sc$ is convex differentiable in $y$ and there exist {$n,\tilde n\in L^2( \mu_T)$} and $m,\tilde m$ with at most polynomial growth so that for all $(\tilde z,s,y)\in (\R^n)^3$, we have  
\begin{align}
\label{eq:bound}
|\Sc\left(\tilde z,s,y \right)| &\leq m(y)+{n(\tilde z,s)} \ , \\
\label{eq:bound_pa}
| \pa_y \Sc\left(\tilde z,s,y \right) | &\leq \tilde m(y)+{\tilde n(\tilde z,s)} \ .
\end{align}

The Monge-Kantorovich duality and the existence of optimal transport map holds for the transport with \eqref{eq:OT} in the sense that 
there exists $\Gamma:\R^n\mapsto \R$ almost everywhere differentiable, $\Gamma^c:(\R^n)^2\mapsto \R$ measurable, and $I:(\R^n)^2\mapsto \R$ measurable satisfying the following conditions. 
\begin{itemize}
    \item[(i)] $\Gamma\in L^1(\nu_T)$ with $\int \Gamma(x)\nu_T(dx)=0$, $\Gamma^c\in L^1( \mu_T)$, and 
\begin{align}\label{eq:conj1}
\Gamma(y)&=\sup_{(\tilde z,s)\in (\R^n)^2}\left\{\Sc\left(\tilde z,s,y \right)-\Gamma^c(\tilde z,s)\right\} \ ,\\
{\Gamma^c(\tilde z,s)}&=\sup_{y\in \R^n}\left\{ \Sc\left(\tilde z,s,y \right)-\Gamma(y)\right\} \ .
\label{eq:conj2}
\end{align}
 \item[(ii)] There exists $\pi^*\in \Pi(  \mu_T,\nu_T)$ such that $\pi^*$ almost surely \begin{align}\label{eq:OTmap}
     Y_T=I(\tilde Z_T,S_T)
 \end{align}
 and recalling that $\int \Gamma(x)\nu_T(dx)=0$, we have
 \begin{align}
\label{eq:OT2}
{\sup_{\pi \in \Pi( \mu_T  , \nu_T)} }\E^\pi[\Sc(\tilde Z_T,S_T,Y_T)]=\E^{\pi^*}[\Sc(\tilde Z_T,S_T,Y_T)]=\int \Gamma^c(\tilde z,s) \mu_T(d\tilde z,ds) \ .
\end{align}

\end{itemize}

\end{assumption}

\begin{rmk}
\begin{enumerate}

\item[(i)] Following the nomenclature of \cite{santambrogio2015optimal}, we call $\Gamma^c$ and $\Gamma$ the Kantorovich potentials. 
We also define $\pi^*_y(d\tilde z,ds)$ the disintegration of the measure $\pi^*$ on $Y_T$, i.e. for all $f$ continuous and bounded, we have
\begin{align}
\label{eq:defdes}
{\E^{\pi^*}[f(\tilde Z_T,S_T,Y_T)]=\int f(\tilde z,s,y)\pi^*_y(d\tilde z,ds)\nu_T(dy)} \ .
\end{align}

\item[(ii)] Although the Assumption \ref{ass.mk} on $(\Sc, \mu_T,\nu_T)$ is implicit, for particular examples of $(\Sc, \mu_T,\nu_T)$ there are many available results in the optimal transport literature that allow us to check that this assumption holds. For example, the examples in Subsubsection \ref{sss:kyle1}-\ref{sss:kyle2} can be handled via the Brenier theorem in \cite{b,m}, whereas one can appeal to 
 \cite[Theorem 10.28 and Theorem 10.38]{villani} for more sophisticated choices of $\Sc$ (of course, under appropriate conditions on $\Sc$). 
 
\item[(iii)] The condition \eqref{eq:bound} is classical in optimal transport theory and yields finiteness of the value of the optimal transport problem. The condition \eqref{eq:bound_pa} is needed to have the admissibility of our equilibrium pricing rule. Under this assumption, the terminal price of the risky asset has a finite second moment.

\end{enumerate}
\end{rmk}

\subsection{Pointwise optimizer for the problem of the informed trader}

We now use the Monge-Kantorovich duality in the sense of Assumption \ref{ass.mk} and define an equilibrium pricing rule. Our construction also exhibits an optimality criterion for the problem of the informed trader. Let $\Gamma$ be as in Assumption \ref{ass.mk} and define the candidate equilibrium pricing rule as 
\begin{align}\label{eq:defH}
H(t,y)=\E[\pa \Gamma(y+Z_T-Z_t)]
\end{align}
and the auxiliary function 
\begin{align}\label{eq:defG}
\Gamma(t,y)=\E[ \Gamma(y+Z_T-Z_t)] \ .
\end{align}

The integrability of \eqref{eq:defH}-\eqref{eq:defG} can be inferred by the Assumption \ref{ass.mk}.
Indeed, as a consequence of \eqref{eq:OT2}, the optimizer in \eqref{eq:conj2} is clearly $y=I(\tilde z,s)$. Thus, by the differentiability of $\Sc$ and $\Gamma$, the first-order optimality condition of \eqref{eq:conj2} easily yields  
\begin{align}\label{eq:derder}
    \pa_y \Gamma(I(\tilde Z_T, S_T))=\pa_y \Sc(\tilde Z_T,S_T,I(\tilde Z_T, S_T))=\pa_x V(Y_T-\tilde Z_T,S_T)
\end{align}
which is square integrable thanks to \eqref{eq:bound_pa} and the fact that $I(\tilde Z_T,S_T)$ has distribution $\nu_T$. Thus, by using \cite[Lemma A.1 and Lemma (3.1)]{back2020optimal}, \eqref{eq:defH}-\eqref{eq:defG} are integrable and 
\begin{align}\label{eq:dergh}
    H(t,y)=\pa_y\Gamma(t,y) \ .
\end{align}
By the condition $\int \Gamma(x)\nu_T(dx)=0$, we also have $\Gamma(0,0)=0$.

The following proposition shows that \eqref{eq:defH} is a pricing rule in the sense of Definition \ref{def:pricing} and provides a simple characterization for optimality of the strategies of the informed trader. 
\begin{proposition}
\label{proposition:foc}
$\Gamma$ is convex and $H$ in (\ref{eq:defH}) is a pricing rule in the sense of Definition \ref{def:pricing}. If the market maker uses this pricing rule 
$P_t=H(t,Y_t)$,
then any choice continuous strategy $X\in \Ac(H)$ of finite variation ensuring 
 \begin{align}\label{eq:target}
Y_T=I(\tilde Z_T,S_T) \ 
\end{align}
is optimal and yields the optimal expected profit 
 \begin{align}
 \label{eq:optimalprofit}
 \E\left[\Gamma^c(\tilde Z_T,S_T)|\Fc_0\right].
\end{align}
\end{proposition}

\begin{proof}
 The convexity of $\Sc$ in $y$ and the representation \eqref{eq:conj1} easily imply that $\Gamma$ is a convex function. 
Due to \eqref{eq:derder} and \eqref{eq:bound_pa}, we have $\E[|\pa_y \Gamma(Z_T)|^2]<\infty$. Thus, we can proceed similarly to \cite[Lemma 3.1]{back2020optimal} and obtain that $H$ in (\ref{eq:defH}) is a pricing rule. 
We now assume that the market maker uses $$P_t=H(t,Y_t)$$
and fix $X\in \Ac(H)$. 
Applying It\^{o}'s formula to $\Gamma(t,Y_t)$, thanks to the heat equation $\pa_t \Gamma(t,Y_t)=-\frac{1}{2} Tr\left(\sigma\sigma^\top \pa_y H(t,Y_t)\right)=-\frac{1}{2} Tr\left(\sigma\sigma^\top \pa^2_y \Gamma(t,Y_t)\right) $, we have
\begin{align*}
    \Gamma(Y_T)&= \int_0^T  P_t^\top dY_t+\int_0^T \pa_t \Gamma(t,Y_t)dt+\frac{1}{2}\int_0^T Tr\left( \pa^2_y \Gamma(t,Y_t)d\langle Y\rangle_t\right) \\
    &=\int_0^T  P^\top_t dX_t+\int_0^T  P^\top_t dZ_t+\frac{1}{2}\int_0^T  Tr\left(\pa^2_y \Gamma(t,Y_t)(d\langle Y\rangle_t-d\langle Z\rangle_t)\right)\\
    &=\int_0^T  P^\top_t dX_t+\int_0^T  P_t^\top dZ_t+\langle P,X\rangle_T-\frac{1}{2}\int_0^T  Tr\left(\pa^2_y \Gamma(t,Y_t)d\langle X\rangle_t\right) \ .
\end{align*}
The condition $X\in \Ac(H)$ implies that $\int_0^\cdot  P_t^\top dZ_t$ is a $\Fc$-martingale. Thus, by the expression of $\Gamma(Y_T)$ and \eqref{eq:conj2}, we have
\begin{align*}
  &\sup_{X\in \Ac(H)} \E\left[ V(\beta+X_T, S_T)  - \int_0^T  P_t^\top dX_t- \langle X,P\rangle_T\ | \ \Fc_0 \right]\\
  &=\sup_{X\in \Ac(H)} \E\left[ \Sc(\tilde Z_T, S_T,Y_T)  - \int_0^T  P_t^\top dX_t- \langle X,P\rangle_T\ | \ \Fc_0 \right]\\
  &=\sup_{X\in \Ac(H)}\E\left[ \Sc(\tilde Z_T, S_T,Y_T) - \Gamma(Y_T)-\frac{1}{2}\int_0^TTr\left( \pa^2_y \Gamma(t,Y_t)d\langle X\rangle_t\right)\ | \ \Fc_0 \right]\\
  &\leq \sup_{X\in \Ac(H)}\E\left[\Gamma^c(\Tilde Z_T,S_T) -\frac{1}{2}\int_0^TTr\left( \pa^2_y \Gamma(t,Y_t)d\langle X\rangle_t\right)\ | \ \Fc_0\right]\leq\E\left[\Gamma^c(\Tilde Z_T,S_T) \ | \ \Fc_0\right] \ ,
\end{align*}
where the first inequality is due to \eqref{eq:conj2} and the last inequality is due to the convexity of $\Gamma$. 
Additionally, for any $X\in \Ac(H)$ which is continuous with finite variation and ensuring $Y_T=I(\tilde Z_T,S_T)$, the two inequalities above are equalities. Indeed, \eqref{eq:conj1}-\eqref{eq:OT2} imply that $\Sc(\tilde Z_T,S_T,I(\tilde Z_T,S_T))=\Gamma(I(\tilde Z_T,S_T))+\Gamma^c(\tilde Z_T,S_T)$ a.s. This shows that any choice of continuous strategy $X\in \Ac(H)$ with finite variation ensuring 
\eqref{eq:target}
is optimal and yields the optimal expected profit $\E\left[\Gamma^c(\Tilde Z_T,S_T) \ | \ \Fc_0\right].$
\end{proof}



\section{Filtering aspects}

In the previous section, we have used the optimal transport problem \eqref{eq:OT} to identify a candidate equilibrium pricing rule \eqref{eq:defH} for the market maker and a candidate target final position \eqref{eq:target} for the informed trader. The fundamental question is whether the target \eqref{eq:target} can be achieved via a strategy $A\in \Ac(H)$ so that the equilibrium holds. Given that we have postulated the pricing rule as \eqref{eq:defH}, $Y$ needs to be a Brownian motion under its own filtration so that the process process is martingale. Thus, inconspicuousness property is in fact a necessary condition for equilibrium (otherwise martingality of the price would be violated). Thus, our aim is to construct an equilibrium strategy $A\in \Ac(H)$ which is also inconspicuous.  We will construct such a strategy in Section \ref{s.bspde}. 

Here, we first analyze the problem of filtering the unobserved state $( \tilde Z_t, S_t )$ conditionally to $\Fc^Y_t$, assuming that the strategy $A_t=A_t(Y_t,\tilde Z_t,S_t)$ of the informed trader satisfies the inconspicuousness property. 

To begin with, let us define $\rho_t(\tilde z,s )$ by
\begin{align}\label{eq:fbspde}
    \rho_t(\tilde z,s)  d\tilde z ds = \P\left(\tilde Z_t \in d\tilde z, S_t\in ds \ | \ \Fc^Y_t \right) 
\end{align}
as the density of the filtered state variable $( \tilde Z_t, S_t )$.

\begin{proposition}
\label{proposition:filtering}If the strategy of the informed trader is inconspicuous, i.e. 
$$\iint  A_t( Y_t, \tilde z,s) \rho_t(\tilde z,s) d\tilde z ds =0,\mbox{ for all }t\in[0,T) \ , $$ then the density $\rho_t$ satisfies the Kushner SPDE
\begin{equation}
\label{eq:kushner}
\begin{split}
{\rm d} \rho_t(\tilde z, s) &= \frac{1}{2} {\rm tr} \!\left({\sigma (\partial_{\tilde z \tilde z} \rho_t(\tilde z,s)) \sigma + \sigma_t (\partial_{ss} \rho_t (\tilde z,s)) \sigma_t}\right) {\rm d}t \\ &\qquad + \left({-\partial_{\tilde z} \rho_t (\tilde z,s) + \rho_t(\tilde z,s)\sigma^{-2{\top}} A_t}\right)^{\top} {\rm d}Y_t
\end{split}
\end{equation}
with initial condition
 \begin{align}
 \label{eq:initcond_filtering}
 \rho_0(\tilde z,s ){\rm d}\tilde z\, {\rm d}\xi   =  \mu_0({\rm d}\tilde z, {\rm d}s) \ .
\end{align}
 Here, $\partial_s \varphi$ denotes the $n \times 1$ gradient and $\partial_{ss} \varphi$ the $n \times n$ Hessian with respect to $s$. Similarly defined with respect to $\tilde z$ are the terms $\partial_{\tilde z} \varphi$ and $\partial_{\tilde z \tilde z} \varphi$.
\end{proposition}

\begin{proof}
Let the test function $\varphi \colon \mathbb{R}^n \times \mathbb{R}^n \mapsto \mathbb{R}$ be twice differentiable in each of its $2n$ arguments. Multivariate It\^o's rule (\cite[Theorem 3.6]{karatzas} and \cite[Remark 4.7.1]{bjork}) gives
\begin{equation}
    \begin{split}
        \varphi(\tilde Z_t, S_t) = \varphi(\tilde Z_0,S_0) &+ \frac{1}{2}\int_0^t {\rm tr} \!\left({\sigma (\partial_{\tilde z \tilde z} \varphi_r(\tilde z,s)) \sigma + \sigma_r (\partial_{ss} \varphi_r (\tilde z,s)) \sigma_r}\right) {\rm d}r \\
        &+ \int_0^t \left({\partial_s \varphi(\tilde Z_r, S_r)}\right)^{\top} \sigma_r {\rm d} W_r + \int_0^t \left({\partial_{\tilde z} \varphi(\tilde Z_r, S_r)}\right)^{\top} \sigma {\rm d} B_r \ .
    \end{split}
\end{equation}
Together with
\begin{equation}
    Y_t = Y_0 + \int_0^t A_r(Y_r, \tilde Z_r, S_r) {\rm d}r + \int_0^t \sigma {\rm d} B_r \ ,
\end{equation}
the pair $\left({\varphi(\tilde Z_t,S_t), Y_t}\right)$ forms a partially observable process. Appealing to \cite[Theorem 8.1]{liptser1977statistics}, we attain the filtering equation
\begin{equation}
\label{filtering-eq}
    \begin{split}
        \E\!\left[ \varphi(\tilde Z_t, S_t)\middle| \mathcal{F}^Y_t\right] &= \E\!\left[ \varphi(\tilde Z_0, S_0)\middle| \mathcal{F}^Y_0\right] \\ &+ \frac{1}{2} \int_0^t \E\!\left[{\rm tr} \!\left({\sigma (\partial_{\tilde z \tilde z} \varphi_r(\tilde z,s)) \sigma + \sigma_r (\partial_{ss} \varphi_r (\tilde z,s)) \sigma_r}\right)\middle| \mathcal{F}^Y_r \right] \rm{d}r
        \\ &+ \int_0^t \E\!\left[ \left({\partial_{\tilde z} \varphi(\tilde Z_r, S_r)}\right)^{\top} \sigma + \varphi(\tilde Z_r, S_r) A^{\top}_r \sigma^{-1} \middle| \mathcal{F}^Y_r\right] \sigma^{-1} {\rm d} Y_r \ .
    \end{split}
\end{equation}
To elaborate our computations, first note that $\E\!\left[ A_t(Y_t, \tilde Z_t, S_t)\middle| \mathcal{F}^Y_t\right] = 0$ due to the inconspicuous trading property, so that the innovation dynamics simplify to $\sigma^{-1} \rm{d}Y_t$. Further, \cite[Corollary 3.40]{bain2009fundamentals} allows us to compute the cross-variation of $\varphi(\tilde Z_t, S_t)$ and $Y_t$:
\begin{equation*}
    \langle \varphi(\tilde Z_{\cdot}, S_{\cdot}), Y\rangle_t = \int_0^t \left({\partial_{\tilde z} \varphi(\tilde Z_r, S_r)}\right)^{\top} \sigma {\rm d}r \ .
\end{equation*}
Recalling the definition of $\rho_t (\tilde z,s)$ in \eqref{eq:fbspde}, 
we rewrite \eqref{filtering-eq} in the integral form
\begin{equation}
    \begin{split}
        &\iint \varphi(\tilde z,s) \rho_t(\tilde z,s) {\rm d} \tilde z \, {\rm d}s = \iint \varphi(\tilde z,s) \rho_0(\tilde z,s) {\rm d}\tilde z \, {\rm d}s \\
        &\qquad + \frac{1}{2}\int_0^t \iint {\rm tr} \!\left({\sigma (\partial_{\tilde z \tilde z} \varphi_r(\tilde z,s)) \sigma + \sigma_r (\partial_{ss} \varphi_r (\tilde z,s)) \sigma_r}\right) \rho_r(\tilde z,s) \, {\rm d}\tilde z \, {\rm d}s \, {\rm d}r \\
        &\qquad + \int_0^t \iint \left({\partial_{\tilde z} \varphi(\tilde z,s) + \varphi(\tilde z,s) \sigma^{-2}A_r}\right)^{\top} \rho_r(\tilde z,s) \, {\rm d}\tilde z \, {\rm d}s \, {\rm d}Y_r \ .
    \end{split}
\end{equation}
Several applications of Fubini's theorem and integration by parts give
\begin{equation}
\label{oWZ1OWTdBe}
    \begin{split}
        &\iint \rho_t(\tilde z,s) \varphi(\tilde z,s) {\rm d}\tilde z \, {\rm d}s = \iint \rho_0(\tilde z,s) \varphi(\tilde z,s) {\rm d} \tilde z \, {\rm d}s \\
        &\qquad + \frac{1}{2}\int_0^t \iint {\rm tr} \!\left({\sigma (\partial_{\tilde z \tilde z} \rho_r(\tilde z,s)) \sigma + \sigma_r (\partial_{ss} \rho_r (\tilde z,s)) \sigma_r}\right) \varphi_r(\tilde z,s) \, {\rm d}\tilde z \, {\rm d}s \, {\rm d}r \\
        &\qquad + \int_0^t \iint \left({-\partial_{\tilde z} \rho(\tilde z,s) + \rho(\tilde z,s) \sigma^{-2}A_r}\right)^{\top} \varphi_r(\tilde z,s) \, {\rm d}\tilde z \, {\rm d}s \, {\rm d}Y_r \ .
    \end{split}
\end{equation}
Since $\varphi(\tilde z,s)$ is arbitrary, \eqref{oWZ1OWTdBe} implies
\begin{equation}
\begin{split}
    \rho_t(\tilde z,s) = \rho_0(\tilde z,s) &+ \frac{1}{2}\int_0^t {\rm tr} \!\left({\sigma (\partial_{\tilde z \tilde z} \rho_r(\tilde z,s)) \sigma + \sigma_r (\partial_{ss} \rho_r (\tilde z,s)) \sigma_r}\right) {\rm d}r \\ &+ \int_0^t \left({-\partial_{\tilde z} \rho_r(\tilde z,s) + \rho_r(\tilde z,s) \sigma^{-2} A_r}\right)^{\top} {\rm d} Y_r,
\end{split}
\end{equation}
which is \eqref{eq:kushner}.
\end{proof}

\section{The problem of the informed trader}\label{s.bspde}
Our objective now is to find an inconspicuous strategy $A$ so that $Y_T=I(\tilde Z_T ,S_T)$. Note that the initial condition of the Kushner equation \eqref{eq:initcond_filtering} is given, whereas our optimality condition $Y_T=I(\tilde Z_T, S_T )$ is a condition at final time $T$. This condition can equivalently be rewritten in terms of $\rho_T$ by requiring that the final measure $\rho_T(\tilde z,s )d\tilde zds  $ has to be supported on the set $\{(\tilde z,s ):I(\tilde z,s )=Y_T\}$. 
Since the optimal coupling $\pi^*$ satisfies $Y_T=I(\tilde Z_T, S_T )$ $\pi^*$-a.s., a sufficient condition on $\rho_T$ is to satisfy
\footnote{We abuse notation here since this measure is singular.}
 \begin{align}
 \label{eq:fincond_filtering}
 \rho_T(\tilde z,s )d\tilde zds  & = \ \pi^*_{Y_T}(d\tilde z,ds ) \ .  
 \end{align}
  Thus, instead of working with the forward SPDE \eqref{eq:kushner}, we work with its backward version, where we postulated the final condition of the BSPDE (backward SPDE or measure-valued BSDE) as \eqref{eq:fincond_filtering} via disintegration $\pi^*_y$ of the optimal coupling $\pi^*$ from the transport problem \eqref{eq:OT}. This leads to the following BSPDE that we need to study
 \begin{align}\label{eq:bspde2}
d\rho_t(\tilde z,s ) &=\frac{1}{2} Tr\left(\sigma^2 \partial_{\tilde z\tilde z }  \rho_t(\tilde z,s )+\sigma^2_t \partial_{ss } \rho_t(\tilde z,s )\right)dt+R_t(\tilde z,s )dY_t \quad \mbox{ for } t<T \ , \\
\label{eq:finalcond}
 \rho_T(\tilde z,s )d\tilde zds  & = \pi^*_{Y_T}(d\tilde z,ds ) \ ,
 \end{align}
where $(Y_t)$ is (any) $\sigma^2$-Brownian motion. Note that as standard in BSDE theory, \cite{el1997backward}, we are free to choose $R_t$ to allow $\rho_t$ to be adapted. Additionally, it is not clear whether the solution of backward equation \eqref{eq:bspde2} defined via the final condition \eqref{eq:finalcond} still satisfies the initial condition \eqref{eq:initcond_filtering}.

Although the equation \eqref{eq:bspde2} is challenging, due to linearity of the equation, the simplest way to study it is through its Fourier transform. 
For $u,v\in \R^n$, denote 
\begin{align}
    \hat \rho_t (u,v):=\int e^{-i (u^\top \tilde z+v^\top s)}\rho_t(\tilde z,s)d\tilde zds \ .
\end{align}
Simple computation leads to the linear BSDE
 \begin{align}\label{eq:bspdehat}
d\hat \rho_t( u,v ) &=-\frac{1}{2} \left(u^\top\sigma^2 u +v^\top \sigma^2_t v \right)  \hat \rho_t(u,v ) dt+\hat R_t(u,v )dY_t \quad \mbox{ for } t<T \ , \\
\label{eq:finalcondhat}
 \hat\rho_T(u,v ) &=\int e^{-i (u^\top \tilde z+v^\top s)}\pi^*_{Y_T}(d\tilde z,ds ) \ ,
 \end{align}
whose explicit solution is $\hat \rho_t( u,v )=\hat \rho_t(Y_t, u,v )$ for a function $\hat \rho_t(y,u,v)$ defined via the equality
 \begin{align}\notag
\hat \rho_t(y, u,v ) &=e^{\frac{1}{2}\int_t^T \left(u^\top\sigma^2 u +v^\top \sigma^2_r v \right)dr } \ \E\left[\int e^{-i (u^\top \tilde z+v^\top s)}\pi^*_{y+Y_T-Y_t}(d\tilde z,ds )\right]\\
&=e^{\frac{1}{2}\int_t^T \left(u^\top\sigma^2 u +v^\top \sigma^2_r v \right)dr } \ \E\left[\int e^{-i (u^\top \tilde z+v^\top s)}\pi^*_{y+Z_T-Z_t}(d\tilde z,ds )\right] \ , \label{eq:bspdehatspol}
 \end{align}
where we use the last line as the definition of $\hat \rho_t$ due to the fact that it only uses $(Z_t)$ which is a Brownian motion we have fixed. 

We now state the main result of the paper which says that if the function $\hat \rho$ is the Fourier transform of a (positive) measure, then one can obtain an explicit expression for an inconspicuous strategy and establish the existence of the equilibrium in the generalized Kyle's model. 
 \begin{thm}\label{thm:main}
Assume that $\hat \rho_t(y,\cdot)$ defined by \eqref{eq:bspdehatspol} is the Fourier transform of a probability measure which has positive density $\rho_t(y,\cdot)$ on $[0,T)$ and $(y,\tilde z,s)\mapsto \rho_t(y,\tilde z,s)$ is $C^{1,2,2}$.
Assume also that the function $A$ 
defined by
\begin{align}\label{def:A}
	 A_t(y, \tilde z,s ) &=\frac{\sigma^2}{\rho_t(y,\tilde z,s )}(\partial_y \rho_t(y,\tilde z,s )+ \partial_{\tilde z} \rho_t(y,\tilde z,s )) \ 
\end{align} 
is so that the Kushner equation for filtering problem of 
\begin{align}\label{eq:deff5}
    \tilde Z_t&=-\beta+Z_t \ , \\
    S_t&=S_0+\int_0^t\sigma_r dW_r\label{eq:deff6}
\end{align}
given the observation 
\begin{align}\label{eq:dynY}
Y_t=\int_0^t A_r(Y_r,\tilde Z_r,S_r)dr+Z_t
\end{align}
admits at most one solution (this assumption also requires that \eqref{eq:dynY} admits a unique strong solution).  

Then, $A\in \Ac(H)$ and necessarily the following initial condition holds
\begin{align}\label{eq:initcond}
    \rho_0(\tilde z,s )d\tilde zds  = \mu_0(d\tilde z,ds  ).
\end{align}
If the informed trader uses the strategy $A$ in \eqref{def:A}, then  $(\rho_t)$ (defined as the inverse Fourier transform of $\hat \rho_t$) is the conditional density of $(\tilde Z_t,S_t)$ conditionally on $\Fc^Y_t$ as in \eqref{eq:fbspde}. Additionally, $A$ in \eqref{def:A} is an inconspicuous optimal strategy against the pricing rule $H$ defined in \eqref{eq:defH} and the pair $(H,A)$ is an equilibrium.
 \end{thm}
\begin{rmk}
(i) The uniqueness of the solution of Kushner equation is a well-studied problem. For example, \cite[Theorem 7.1]{bhatt1995uniqueness}\footnote{Condition (i) in this theorem should read $|(c,LQ_\infty)|\leq c |Q^{1/2}x||Q^{1/2}y|$.} allows us to claim this uniqueness holds if $A$ is the sum of a linear function and bounded Lipschitz continuous function of $(y,\tilde z,s)$. This is in fact the case for many cases considered in the literature for Kyle model. 

(ii) The Theorem \ref{thm:main} is surprising in the sense that the solution of the BSPDE \eqref{eq:bspde2} defined via the final condition \eqref{eq:finalcond} also satisfies the initial condition \eqref{eq:initcond}. Thus, via the inconspicuousness property, and using the BSPDE \eqref{eq:bspde2}-\eqref{eq:finalcond}, we construct together $A$ and a solution to the forward SPDE \eqref{eq:kushner} with initial condition \eqref{eq:initcond_filtering} (of course up being able to proceed to perform inverse Fourier transform of \eqref{eq:bspdehatspol}). With this procedure, we have constructed a type of two-point boundary value problem in the space of probability measures, see \cite{ma2022generalized}.

(iii) The Theorem \ref{thm:main} shows how to use the optimal transport problem \eqref{eq:OT} to construct the equilibrium strategy of the informed trader via $\pi^*$ and \eqref{eq:bspdehatspol}. Thus, we have a constructive methodology to establish the existence of equilibrium.

(iv) The solution of the BSPDE \eqref{eq:bspde2} also allows us to explicitly compute $A$ via the expression \eqref{def:A}. In fact, this construction is a generalization of the Doob's h-transform. Indeed, in the classical Doob's h-transform, the additional drift term that is needed to be applied to the process due to conditioning is a cross variation that can be written as 
$$\frac{\sigma^2\partial_y \rho_t(y,s )}{\rho_t(y,s )} $$
when the conditioning is independent of $\tilde z$, \cite[IV.39-40]{rogers2000diffusions}. In our context, since $d\tilde Z_t=dZ_t$ appears in both the dynamics of the observed state $Y$ and the unobserved state $\tilde Z$, the impact  on the Doob's h-transform of this correlated noise in the filtering problem can be obtained by comparing \eqref{eq:kushner} and \eqref{eq:bspde2}.

(v) The Assumption that $\rho_t>0$ is in fact the most challenging issue we face to state a general existence of equilibrium result. Indeed, thanks to \eqref{eq:bspdehatspol}, it is fairly simple to solve \eqref{eq:bspde2} in the Fourier space (as linear BSDEs). In fact, for $n=1$, denoting the sine cardinal function
$sinc(x)=\frac{\sin \pi x}{\pi x}$, we can convolve $\pi^*_y(d\tilde z,ds)$ with the density 
$$\frac{sinc^2\left(\frac{\tilde z}{\varepsilon}\right)sinc^2\left(\frac{s}{\varepsilon}\right)}{(\varepsilon \int sinc^2(u)du)^2} $$
for some $\varepsilon>0$. Thus, thanks to the convolution theorem and the fact that the Fourier transform of $x \mapsto \frac{1}{\varepsilon \int sinc^2(u)du} sinc^2\left(\frac{x}{\varepsilon}\right)$ is the triangular tent function, we can approximate $\hat\rho_T(u,v )$ with $\hat \rho^\varepsilon_T(u,v)$ which is $0$ for $u,v$ large. This would mean that (the compactly supported function)
$$\hat \rho_t^\varepsilon(y,u,v)=e^{\frac{1}{2}\int_t^T \left(u^\top\sigma^2 u +v^\top \sigma^2_r v \right)dr }\E\left[\int e^{-i (u^\top \tilde z+v^\top s)}\pi^\varepsilon_{y+Z_T-Z_t}(d\tilde z,ds )\right]$$  is indeed the Fourier transform of a smooth map $\rho^\varepsilon_t(y,\cdot)$. Thus, we can easily construct  smooth approximate solutions to \eqref{eq:bspde2}, where we approximate the final condition. However, we have no way of checking that  $\rho^\varepsilon_t(y,\cdot)\geq 0$ (the BSPDE \eqref{eq:bspde2} does not satisfy a comparison principle). Thus, we are not able to define $A$ everywhere due to the division by $\rho_t$. 

In Subsection \ref{ss.counter}, we provide a counter-example showing that one cannot expect a general existence result and our methodology has to be applied on case-by-case basis. Despite this limitation, the methodology we present allows the proof of existence of equilibrium in all cases of Kyle's model studied in the literature.

(vi) A combination of equalities \eqref{eq:defH}-\eqref{eq:derder} and \eqref{def:A} shows that both the strategy of the informed trader and the pricing rule of the market maker can be written as sensitivities of relevant value functionals in the innovation process $Y$. The relevant functional for the market maker is $\Gamma(t,y)$ and for the informed trader the entropic term $\sigma^2 \ln \rho(t,Y_t,\tilde z,s)$.

(vii) Our methodology, which consists of choosing $A$ so that we have an explicit expression to the Kushner equation, allows us to claim that $Y$ is a Brownian motion under its own filtration without checking the conditions in \cite{follmer1999canonical} which would have required us to postulate a functional form for $A$.


\end{rmk}

\begin{proof}[Proof of Theorem \ref{thm:main}] Define $\hat \rho_t(y,u,v)$ by \eqref{eq:bspdehatspol} and assume that $\hat \rho_t(y,\cdot)$ is the Fourier transform of a probability measure which has positive density $\rho_t(y,\cdot)$ on $[0,T)$. 
Since the Fourier transform identifies probability measures, it is sufficient to show that 
\begin{align*}
\hat \rho_0(0,u,v)&:=e^{\frac{1}{2}\int_0^T \left(u^\top\sigma^2 u +v^\top \sigma^2_r v \right)dr }\E\left[\int e^{-i (u^\top \tilde z+v^\top s)}\pi^*_{Z_T}(d\tilde z,ds )\right] \\
&=\int e^{-i (u^\top \tilde z+v^\top s)}\mu_0(d\tilde z,ds) \ .
\end{align*}
Note that $\pi^*$ is the optimal coupling between $\mu_T$ and $\nu_T$ and $\pi^*_y$ is the disintegration of this measure according to $y$ (which has distribution $\nu_T$). Since $Z_T$ also has distribution $\nu_T$, it is clear that 
$$\E\left[\int e^{-i (u^\top \tilde z+v^\top s)}\pi^*_{Z_T}(d\tilde z,ds )\right]=\int e^{-i (u^\top \tilde z+v^\top s)}\pi^*(dy,d\tilde z,ds ) \ .$$
Thus, 
\begin{align*}
    \hat \rho_0(0,u,v)&:=e^{\frac{1}{2}\int_0^T \left(u^\top\sigma^2 u +v^\top \sigma^2_r v \right)dr }\E\left[e^{-i (u^\top \tilde Z_T+v^\top S_T)}\right]\\
    &=e^{\frac{1}{2}\int_0^T \left(u^\top\sigma^2 u +v^\top \sigma^2_r v \right)dr }\E\left[e^{-i (u^\top Z_T+v^\top \int_0^t \sigma_r dW_r)}\right]\E\left[e^{-i (- u^\top \beta+v^\top S_0)}\right]\\
    &=\E\left[e^{-i (- u^\top \beta+v^\top S_0)}\right] \ .
\end{align*}
Thus, $\hat \rho_0(0,u,v)$ is indeed the Fourier transform of the distribution of $(-\beta,S_0)$ and \eqref{eq:initcond} holds. 

Due to \eqref{eq:dynY}, in order to show that $A\in \Ac(H)$, it is sufficient to show that 
$$ \int_0^T \E[|H|^2(s,Y_s)]ds<\infty \ .$$

Since $\rho$ is smooth and $\hat \rho$ solves \eqref{eq:bspdehat}, $\rho$ solves \eqref{eq:bspde2} with $R_t(\tilde z,s)=\pa_y \rho_t(Y_t,\tilde z,s)$. Given the expression of $A$ and the initial condition \eqref{eq:initcond}, $\rho$ solves the forward SPDE \eqref{eq:kushner}. 

Denote $\tilde \rho_t$ the conditional density of $(\tilde Z_t,S_t)$ conditional on $\Fc^Y_t$ for the system \eqref{eq:deff5}-\eqref{eq:dynY} which also solves the forward SPDE \eqref{eq:kushner}. By assumption of uniqueness of solutions to the Kushner equation, we have $\tilde \rho_t=\rho_t(Y_t,\cdot).$

Then, $A$ satisfies
 \begin{align*}
\iint  A_t( Y_t,\tilde z , s  ) \rho_t(Y_t,\tilde z , s  )d\tilde z  d s    
&=\sigma^2  \iint \big( \pa_y \rho_t(Y_t,\tilde z , s  ) +\partial_{\tilde z }   \rho_t(Y_t,{\tilde z }  , s  ) \big) \ d{\tilde z }   d s   \\
&=\sigma^2\pa_y\iint  \rho_t(Y_t,{\tilde z }  , s  ) \ d{\tilde z }   d s  =\sigma^2\pa_y 1 =0 \ ,
\end{align*}
and is therefore inconspicuous. Thanks to Proposition \ref{prop:inco}, we obtain that $Y$ has the same distribution as $Z$ and by the definition \eqref{eq:defH}, $H(s,Y_s)$ is a martingale which  leads to $$ \int_0^T \E[|H|^2(r,Y_r)]dr<\infty \ .$$
Thus, $X_t=\int_0^t A_r(Y_r,\tilde Z_r,S_r)dr$ is indeed in $\Ac(H)$.

$\rho_T$ also satisfies \eqref{eq:finalcond} which implies that $Y_T=I(\tilde Z_T,S_T)$ and thanks to Proposition \ref{proposition:foc}, $X$ is indeed optimal and inconspicuous. Thus, to conclude the equilibrium property of $(H,X)$, we need to show that the martingale 
$H(t,Y_t)$ has terminal condition 
$$H(T,Y_T)= \pa_x V(X_T+\beta,S_T)=\pa_y\Sc (\tilde Z_T,Y_T,S_T) \ .$$
This equality is a consequence of \eqref{eq:derder} and \eqref{eq:dergh} and the fact that at terminal time the joint law of $(\tilde Z_T,S_T,Y_T)$ is $\pi^*$ which satisfies 
$$Y_T=I(\tilde Z_T,S_T),\,\pi^*-\mbox{a.s.}$$
\end{proof}

\section{Examples}\label{s.examples}
\subsection{Optimal transport map is independent of $\tilde z$}
In this Subsection, we treat the examples mentioned in the Subsubsections \ref{sss:kyle1}-\ref{sss:kyle2}.
Similar to \cite{ccd}, we assume that $S_0$ normal with $0$ mean and variance $\Sigma_0^2$\footnote{With notations of \cite{ccd}, $Z_0$ in \cite{ccd} is assumed to have distribution $G(0,0;c,\cdot)$. Given our discussion in Subsubsection \ref{sss:kyle2}, this implies in our framework that $S_0$ is Gaussian.}. Denote the covariance matrix of $S_T$ as
$$\Sigma_T^2=\Sigma_0^2+\int_0^T \sigma_r^2dr$$
and $\lambda=\frac{\sigma^{-1}(\sigma\Sigma_T^2\sigma)^{1/2}\sigma^{-1}}{\sqrt{T}}$ so that $y\mapsto \lambda y$ is the optimal transport map from the Gaussian law $\nu_T$ of $Y_T$ to the Gaussian law $\mu^2_T$ of $S_T$. We assume that $y\mapsto f(\lambda y)$ is the gradient of a convex function (which is always the case in the one dimensional case of \cite{back1998long,ccd}).

For the models mentioned in subsubsections \ref{sss:kyle1}-\ref{sss:kyle2}, the function $V$ is linear in $x$ and the optimal transport problem \eqref{eq:OT} reduces to
\begin{align*}
\sup_{\pi \in \Pi( \mu_T  , \nu_T)} \E^\pi[(Y_T-\tilde Z_T)^\top f(S_T)]=-\E[\tilde Z_T^\top f(S_T)]+\sup_{\pi \in \Pi( \mu^2_T  , \nu_T)} \E^\pi[Y_T^\top f(S_T)] \ .
\end{align*}
Here, $\Pi( \mu^2_T  , \nu_T)$ is the set of couplings of $(Y_T,S_T)$ in which the distribution of $Y_T$ is $\nu_T$ and the distribution of $S_T$ is the measure $\mu^2_T .$
Using the Monge-Kantorovich duality, it is easy to see that the optimizer of 
$$\sup_{\pi \in \Pi( \mu^2_T  , \nu_T)} \E^\pi[Y^\top_T f(S_T)]$$
is achieved at a measure $\pi^*$ which is supported on 
$$\pa \Gamma(Y_T)=f(S_T)$$
for a function $\Gamma$ satisfying 
$$\Gamma(y)=\sup_{s}\left\{y^\top f(s)-\Gamma^c(s)\right\}$$
and is therefore convex. Due to the assumption that $y\mapsto f(\lambda y)$ is convex, we have that $y\mapsto f^{-1}(\pa\Gamma(y))$ is the optimal transport map (for the Wasserstein-2 distance) from $\nu_T$ to $\mu^2_T$. Note that these distributions are Gaussian and this transport map is $y\mapsto \lambda y$ so we have
\begin{align}\label{eq:defpg}
    f^{-1}(\pa\Gamma(y))=\lambda  y \ .
\end{align}
There exists a measurable function $I$ such that $Y_T=I(f(S_T))$ $\pi^*$-a.s.
Thus, the BSDEs \eqref{eq:bspdehat}-\eqref{eq:finalcondhat} reduces to 
\begin{align*}
d\hat \rho_t( v ) &=-\frac{1}{2} \left(v^\top \sigma^2_t v \right)  \hat \rho_t(v ) dt+\hat R_t(v )dY_t \quad \mbox{ for } t<T \ , \\
\hat\rho_T(v ) &= e^{-i v^\top f^{-1}(\pa \Gamma(Y_T))}=e^{-i v^\top \lambda  Y_T} \ .
\end{align*}
whose explicit solution \eqref{eq:bspdehatspol} is
 \begin{align*}
\hat \rho_t(y, u,v ) &=e^{\frac{1}{2}\int_t^T v^\top \sigma^2_r v dr }\E\left[e^{-i v^\top \lambda  (y+Z_T-Z_t)}\right]\\
&=e^{-i v^\top \lambda  y+\frac{1}{2}\int_t^T ( |\sigma_r v|^2-|\lambda \sigma v|^2 ) dr } \ .
 \end{align*}
In this case, we have the negativity condition
\begin{align}
\int_t^T ( \sigma_r^2-\sigma \lambda^2\sigma )  dr&=\int_t^T  \sigma_r^2dr -(T-t)\sigma \lambda ^2 \sigma\notag\\
&=\int_t^T  \sigma_r^2dr -\frac{T-t}{T}\Sigma_T^2 < 0 \label{eq:condsigmar}
\end{align}
for $t<T$ (in the sense of symmetric matrices), 
$$\rho_t(y,\tilde z,\cdot)=\rho_t(y,\cdot)=\mbox{density of }N\left(\lambda y,\frac{T-t}{T}\Sigma_T^2-\int_t^T  \sigma_r^2dr \right)$$
is a solution to 
\eqref{eq:bspde2} (without the dependence on $\tilde z$ of the equation). Thanks to \eqref{def:A}, the optimal strategy for the informed trader is
 \begin{align}
	 A_t( Y_t, s)
	 &=\frac{\sigma^2}{\rho_t(Y_t,s)}(\partial_y \rho_t(Y_t,s)+ \partial_z \rho_t(Y_t,s)) \\
	 &= \sigma^2 \lambda  \left({ \frac{T-t}{T} \Sigma_T^2 - \int_t^T \sigma^2_r \, {\rm d}r }\right)^{-1} (s - \lambda Y_t) \ .
\end{align}
This is a multidimensional generalization of \cite{back1998long}, for whom $n=1$, $\Sigma_T = \sigma$\footnote{again, we note \cite{back1998long} permits the volatility of the noise trades to vary with time}, $T = 1$ (implying $\lambda  = 1$), and
\begin{equation}
A_t(Y_t,s) = \sigma^2 \frac{s - Y_t}{(1-t) \sigma^2 - \int_t^1 \sigma^2(u) \, {\rm d} u} \ .
\end{equation}
Also compare this result to \cite[Theorem 5.1 (ii)]{ccd}, in which $\Sigma_T = \sigma = 1$ by virtue of their assumption 2.1. Their strategy can be written explicitly as
\begin{equation}
    A(t,Y_t,s) = \frac{s - Y_t}{(1 - t) - \int_t^1 \sigma^2(u)\, {\rm d}u} \ .
\end{equation}

\eqref{eq:defpg} also shows that $\pa \Gamma(y)=f(\lambda y)$. Thanks to \eqref{eq:dergh}, we obtain the equilibrium pricing rule
$$H(t,y)=\E[f(\lambda (y+Z_T-Z_t))] \ .$$

In the following Lemma, we summarize our existence of  equilibrium in  Kyle-Back model using our approach. 

\begin{lemma}[Kyle-Back model with dynamic information]\label{example:Kyle-Back}
    Assume $S_0 \sim N(0,\Sigma_0^2)$ and denote the covariance matrix of $S_T$ as
$$\Sigma_T^2=\Sigma_0^2+\int_0^T \sigma_r^2dr \ . $$
We define $\lambda  :=\frac{\sigma^{-1}(\sigma\Sigma_T^2\sigma)^{1/2}\sigma^{-1}}{\sqrt{T}}$ and consider $V(x,s)=x^\top f(s)$, which gives
    \begin{equation}
        \sup_{\pi \in \Pi(\mu_T,\nu_T)}\E^\pi[\Sc(\tilde Z_T,S_T,Y_T)] = - \E[ \tilde Z_T^\top f(S_T)] +  \sup_{\pi \in \Pi(\mu_T^2,\nu_T)}\E^\pi[Y_T^\top f(S_T)] \ .
    \end{equation}
     Then, we obtain the following ingredients for the equilibrium.
    \begin{itemize}
        \item[(a)]
        Potentials are
    \begin{align*}
        \Gamma(y) = \frac{1}{2} y^\top \lambda y \ , \quad
        \Gamma^c(\tilde z,s) = -\tilde z^\top s + \frac{1}{2} s^\top \lambda  s \ .
    \end{align*}
        
        \item[(b)] Optimal transport map is
        $$y = I(\tilde z, s) = I(s) = {\lambda ^{-1}}s \ .$$
        
        \item[(c)] Pricing rule is
        $$H(t,Y_t)= \E[f(\lambda (Y_t+Z_T-Z_t))] \ .$$
        
        \item[(d)] The optimal strategy for the informed trader is
            \begin{align}\label{eq:optimalStrat:multiD_KB}
                A_t(Y_t, \tilde Z_t, S_t) = A_t(Y_t, S_t) 
			=  \sigma^2 \lambda  \left({ \frac{T-t}{T} \Sigma_T^2 - \int_t^T \sigma^2_r \, {\rm d}r }\right)^{-1} (S_t - \lambda Y_t) \ .
            \end{align}
    \end{itemize}
\end{lemma}

\subsection{Surplus function has no dependence on $s$} 
In this subsection, we will take a look at examples with surplus function of the form
\begin{equation*}
    \Sc (\tilde z, y) = V(y-\tilde z, s) = V(y-\tilde z) ,
\end{equation*}
so that the optimal transport problem in hands becomes
\begin{equation}
    \sup_{\pi \in \Pi(\mu^1_T, \nu_T)}\E^\pi[V(Y_T -\tilde Z_T)],
\end{equation}
where $\tilde Z_T =  Z_T - \beta \sim \mu^1_T$ and $Y_T\sim \nu_T=N(0,\sigma^2T)$. That is to say, our informed trader is a potential activist who has no private signal $S_T$ on the asset, but has private information on her position. For $t<T$, she can profitably trade on her private information about the number of shares she owns.

We assume that the function $V(x,s)=V(x)$ is strictly convex and $n=1$. Then, in this one dimensional example, \cite[Remark 2.12]{santambrogio2015optimal} implies that the optimal transport map does not depend on $V$ and is the \textit{unique} strictly decreasing coupling between the Gaussian distribution $\mu_T^1$ and $\nu_T$. If $V$ is not strictly convex but only convex, we can use \cite[Lemma 2.10]{santambrogio2015optimal} and stability of optimal transport to show the optimality of the monotone decreasing transport map. 

We assume that $\beta$ is $N(m_\beta, \sigma_\beta^2)$ so that $\mu^1_T=N(-m_\beta,\sigma_\beta^2+\sigma^2 T)$.
Also, we let $$\lambda^2 = \frac{\Var \tilde Z_T}{\Var Y_T} = 1 + \frac{\sigma_\beta^2}{\sigma^2 T} \ .$$
Thus, the optimal transport map (the unique decreasing map form the distribution of $\tilde Z_T$ to the distribution of $Y_T$) is 
\begin{align}\label{example:activism:OTmap}
  y = I(\tilde z) = -\frac{1}{\lambda} (\tilde z + m_\beta) \ . 
\end{align}
Using the first order optimality condition in \eqref{eq:conj1}-\eqref{eq:conj2} and given the optimal transport map, the potentials $\Gamma(y),\Gamma^c(z)$ has to satisfy
$$\Gamma'(y)=V'(y-I^{-1}(y))\mbox{ and }(\Gamma^c)'(\tilde z)=-V'(I(\tilde z)-\tilde z).$$
This allows us to identify the potentials (up to an additive constant) as
\begin{align}\label{pot1eq2}
    \Gamma(y)& =\frac{ V(y(1+\lambda)+m_\beta)}{\lambda +1} \ ,\\
    \Gamma^c(z)& =\frac{V(-(1+\lambda^{-1})z-\lambda^{-1}m_\beta)}{1+\lambda^{-1}}\label{pot2eq2} \ .
\end{align}

Thus, for the particular case of $n=1$ and appropriate choices of $V$, we recover the problem studied in \cite{bal}. 
We summarize this in the following Lemma. 
\begin{lemma}
\label{example:activism}
Let $n=1$ and $\beta$ be Gaussian with mean $m_\beta$ and variance $\sigma_\beta^2$. Assume that $V$ is convex and \eqref{eq:bound}-\eqref{eq:bound_pa} are satisfied. 
\begin{enumerate}
    \item[(a)] Kantorovich potentials for the optimal transport problem are given by \eqref{pot1eq2}-\eqref{pot2eq2},
and the optimal transport map is given by \eqref{example:activism:OTmap}.

    \item[(b)] The pricing rule is
\begin{align*}
    H(t,Y_t) = \E[V'((1+\lambda) Y_T + m_\beta) \mid \mathcal{F}_t^Y] \ .
\end{align*}

\item[(c)] The optimal strategy $A_t(y,\tilde z)$ for the informed trader is given as
\begin{align}
\label{example:activism:OptimalStrat}
    A_t(Y_t,\tilde Z_t)= - \frac{1}{T-t} \Big( \frac{\tilde Z_t +  m_\beta }{\lambda - 1} - \frac{\lambda}{\lambda - 1}Y_t \Big) \ ,
\end{align}
and the equilibrium position at time $T$ is
\begin{align*}
    X_T + \beta = m_\beta + (\lambda+1) Y_T \ .
\end{align*}
\end{enumerate}
\end{lemma}
\begin{rmk}
The reason we are unable to state this result for general $n$ is the fact that the linearity of the optimal transport map only trivially holds for $n=1$. For particular examples of $(V,\beta)$ for $n\geq 1$, if the optimal transport map $\tilde z\mapsto  I(\tilde z)$ turns out to be linear, then our methodology would lead to the existence of equilibrium.  
\end{rmk}

\subsection{Surplus function depends on $\tilde z, s$ in a linear-quadratic way.}
In this subsection, we show that results in \cite{collin2015shareholder} can be recovered using our methodology. Although our methodology allows us to treat for general $n$, for notational simplicity we assume $n=1$. Note that we are also able to treat the dynamic information model ($\sigma_t>0$). 

Consider $ V(x,s) =\frac{\psi}{2}x^2 + x s$ for $\psi>0$, which gives
\begin{align*}
    \Sc(\tilde z,s,y) =V(y-\tilde z, s) & = \frac{\psi}{2}(y-\tilde z)^2 +(y- \tilde z) s \\
    & =\frac{\psi}{2}(y^2+ \tilde z^2 )- \tilde z s + y (s- \psi \tilde z) \ .
\end{align*}
So, the optimal transport problem \eqref{eq:OT} reduces to
\begin{align*}
    \sup_{\pi \in \Pi( \mu_T  , \nu_T)} \E^\pi[\Sc(\tilde Z_T, S_T, Y_T)]
    =\frac{\psi}{2}(Y_T^2 + \tilde Z_T^2 )- \tilde Z_T S_T 
    +\sup_{\pi \in \Pi( \mu_T^u  , \nu_T)} \E^\pi[Y_T U_T] \ ,
\end{align*}
where $U_T := S_T- \psi \tilde Z_T$ has distribution $\mu_T^u $.
Let us assume $\beta$ is $N(m_\beta, \sigma_\beta^2)$ and the terminal signal $S_T$ on the stock is $N(0, \Sigma_T^2)$ independent of $\beta$. Then, we have $\mu_T^u = N(\psi m_\beta,\psi^2(\sigma^2 T + \sigma_\beta^2) + \Sigma_T^2) $, and we can explicitly obtain optimal transport map, pricing rule, and the strategy for the informed trader.

\begin{lemma}
\label{example:gaussian}
Assume $S_0 \sim N(0,\Sigma_0^2)$ independent from $\beta\sim N(m_\beta, \sigma_\beta^2)$ and denote the covariance matrix of $S_T$ as
$$\Sigma_T^2=\Sigma_0^2+\int_0^T \sigma_r^2dr \ . $$
Consider $ V(x,s) =\frac{\psi}{2}x^2 + x s$ for $\psi>0$, and let us write
\begin{align*}
     \tilde \sigma_t^2 &:= \Var \tilde Z_t = \sigma^2 t + \sigma_\beta^2, \qquad
\lambda^2= \frac{\Var S_T}{\Var Y_T} = \frac{\Sigma_T^2}{\sigma^2 T}, \qquad
\tilde \lambda^2 = \frac{\Var \tilde Z_T}{\Var Y_T} = \frac{\tilde\sigma_T^2}{\sigma^2 T}\\
K_t&:=(\sigma^2 \lambda^2 T - \Sigma_T^2 + \Sigma_t^2)(\sigma^2 \tilde\lambda^2T - \sigma^2(T-t))\\
&-\varepsilon^2\sigma^2\psi^2\tilde \lambda^4 t(\sigma^2 \lambda^2 T - \Sigma_T^2 + \Sigma_t^2)-\varepsilon^2\sigma^2\lambda^4t(\sigma^2 \tilde\lambda^2T - \sigma^2(T-t)) \ , \\
\varepsilon &= \frac{1}{\sqrt{\psi^2 \tilde\lambda^2 + \lambda^2}}, \qquad
\mathbf{u} = \varepsilon \begin{bmatrix}
- \psi \tilde\sigma^2_T \\  \Sigma_T^2
\end{bmatrix}.
\end{align*}
Assume that $K_t>0 $ and $\sigma^2 \tilde\lambda^2(T-\varepsilon^2 \psi^2 \tilde\lambda^2 t) - \sigma^2(T-t)> 0$ for all $t\in[0,T)$. 
Then, there exists an equilibrium in the generalized Kyle model with the following properties.  
	\begin{enumerate}
		\item[(a)] We have potential (up to an additive constant) and the optimal transport map as
			\begin{align*}
				\Gamma(y) &= \frac{1}{2}\left( \psi y^2 + 
\frac{1}{\varepsilon} (y + \varepsilon \psi m_\beta)^2
\right) \ , \\
				I(\tilde z,s) & = \varepsilon (s - \psi (\tilde z + m_\beta)) \ . 
			\end{align*}
			
		\item[(b)] In equilibrium, the conditional distribution of $(\tilde Z_T , S_T)$ given $Y_T = y$ is 
			\begin{align*}
				N\left(\begin{bmatrix}
-m_\beta \\ 0
\end{bmatrix}+ {y}{\varepsilon} \begin{bmatrix}
-\psi \tilde\lambda^2 \\  \lambda^2
\end{bmatrix}, \ 
				 \varepsilon^2 \lambda^2 \tilde\lambda^2 \sigma^2 T
\begin{bmatrix}
 1 \\  \psi
\end{bmatrix} \begin{bmatrix}
 1 & \psi
\end{bmatrix}  \right) \ .
			\end{align*}
		\item[(c)] The optimal strategy $A_t(y,\tilde z,s)$ for the informed trader is given as
		 	\begin{align}\label{eq:psia}
		 		A_t(y,\tilde z,s) = A_0 (t) y + A_1 (t) (\tilde z + m_\beta) + A_2 (t) s \ ,
			\end{align}
			where
\begin{align*}
    A_0(t) & =-\frac{\sigma^2}{K_t} \Big(\varepsilon\psi\tilde \lambda^2(\sigma^2 \lambda^2T-\Sigma_T^2+\Sigma_t^2)(1+\varepsilon\psi \tilde \lambda^2)+\sigma^2\varepsilon^2\lambda^4((\tilde \lambda^2-1) T+t)\Big) \\
    A_1(t) &=-\frac{\sigma^2}{K_t} \Big((1+ \varepsilon\psi \tilde \lambda^2)(\sigma^2\lambda^2(T- \varepsilon^2\lambda^2 t)-\Sigma^2_T+\Sigma^2_t)+ \varepsilon^2\psi\tilde \lambda^2\lambda^4\sigma^2t\Big)\\
    A_2(t) &=\frac{\sigma^2}{K_t} \Big((1+ \varepsilon\psi \tilde \lambda^2)(\varepsilon^2\psi\lambda^2\tilde\lambda^2\sigma^2t)+\varepsilon\lambda^2\sigma^2(\tilde\lambda^2(T-\varepsilon^2\psi^2\tilde \lambda^2t)-T+t)\Big).
\end{align*}
	\end{enumerate}
\end{lemma}

\begin{proof} 

(a) : The potentials are 
\begin{align*}
\Gamma(y)&=\frac{1}{2}\left( \psi y^2 + 
\frac{1}{\varepsilon} (y + \varepsilon \psi m_\beta)^2
\right) \ , \\
\Gamma^c(\tilde z,s) & =\frac{\psi}{2}\tilde z^2 -  \tilde z s
- \varepsilon \psi m_\beta (s-\psi \tilde z) + 
\frac{\varepsilon}{2} (s-\psi \tilde z)^2.
\end{align*}
Indeed, $\epsilon-$Young inequality yields that 
\begin{align*}
\Gamma(y)+\Gamma^c(\tilde z,s)&= \frac{\psi}{2} (y^2 + \tilde z^2) -\tilde z s
- \varepsilon \psi m_\beta (s-\psi \tilde z)\\
& \qquad+\frac{1}{2}\left(\frac{1}{\sqrt{\varepsilon}}(y+ \varepsilon  \psi m_\beta)\right)^2+\frac{1}{2}\left(\sqrt{\varepsilon}(s - \psi \tilde z) \right)^2\\
&\geq \Sc(\tilde z,s,y)
\end{align*}
with equality if 
\begin{align}
\label{eq:OTmapExample}
    y=\varepsilon (s - \psi (\tilde z + m_\beta))=:I(\tilde z,s) \ .
\end{align}
We can also prove they are potentials and optimal transport map by \cite[Theorem 1.47]{santambrogio2015optimal}. 
Indeed, the map $y \mapsto (\pa_{\tilde z} \Sc(\tilde z,s,y),\pa_s\Sc(\tilde z,s,y) ) = (\psi(\tilde z-y)-s, y-\tilde z)$ is injective for all $(\tilde z,s)$. 
Also, as a consequence of the $\epsilon-$Young inequality, one can check that $\Gamma^c(\tilde z, s) = \sup_y \Big( \Sc(\tilde z, s,y) - \Gamma(y) \Big)$ with the supremum being achieved at $y=I(\tilde z, s)$. By the envelope theorem, we have $\pa_{\tilde z, s} \Sc( \tilde z, s,I(\tilde z, s)) = \pa \Gamma^c (\tilde z, s)$. 
Moreover, the law of the map $I(\tilde Z_T, S_T)$ in \eqref{eq:OTmapExample} is Gaussian since it is linear combination of Gaussians, and it has mean $0$ and variance $\sigma^2 T$. Hence,  $I(\tilde Z_T, S_T)$ is an optimal transport map between  $ \mu_T$ and $\nu_T$. 

(b) : Let us compute the distribution of $(\tilde Z_T, S_T)$ given $Y_T = y$ when $Y_T=I(\tilde Z_T,S_T)$. The vector 
$$(\tilde Z_T, S_T, Y_T)=(\tilde Z_T, S_T, I(\tilde Z_T, S_T))$$ is Gaussian with mean $(-m_\beta,0,0)$ and covariance matrix 
$$ 
\boldsymbol{\sigma} = \begin{bmatrix}
\tilde\sigma^2_T & 0 & -\varepsilon \psi \tilde\sigma^2_T\\
0 & \Sigma_T^2 & \varepsilon  \Sigma_T^2 \\
-\varepsilon \psi \tilde\sigma^2_T &\varepsilon   \Sigma_T^2 & \sigma^2 T
\end{bmatrix} \ .
$$
By Schur's complement formula, the conditional mean is
\begin{align*}
\mathbb{E}[ (\tilde Z_T, S_T)^\top \ | \ Y_T = y ] & = 
\begin{bmatrix}
-m_\beta \\ 0
\end{bmatrix}+ \frac{y}{\sigma^2 T} \mathbf{u} 
= \begin{bmatrix}
-m_\beta \\ 0
\end{bmatrix}+ {y}{\varepsilon} \begin{bmatrix}
-\psi \tilde\lambda^2 \\  \lambda^2
\end{bmatrix} \ ,
\end{align*}
and the covariance matrix is
\begin{align*}
\Cov\left( (\tilde Z_T, S_T)^\top \ | \ Y_T = y \right) & = \begin{bmatrix}
\tilde\sigma^2_T  &         0    \\
0  &       \Sigma_T^2  \\
\end{bmatrix} 
- \frac{1}{\sigma^2 T} \mathbf{u} \mathbf{u} ^T \\
& = \begin{bmatrix}
\tilde\sigma^2_T  &         0    \\
0  &       \Sigma_T^2  \\
\end{bmatrix}  - \frac{1}{\psi^2 \tilde\sigma^2_T + \Sigma^2_T} \begin{bmatrix}
\psi^2 \tilde\sigma_T^4 & -\psi \tilde\sigma_T^2 \Sigma_T^2 \\ -\psi \tilde\sigma_T^2 \Sigma_T^2  &  \Sigma_T^4 
\end{bmatrix} \\
& = \varepsilon^2 \lambda^2 \tilde\lambda^2 \sigma^2 T
\begin{bmatrix}
 1 \\  \psi
\end{bmatrix} \begin{bmatrix}
 1 & \psi
\end{bmatrix} 
\end{align*}
which yields the result.

(c) : 
By taking a Fourier transform in both $\tilde z$ and $s$, \eqref{eq:bspdehatspol} becomes 
\begin{align*}
 \hat \rho_t(y,u,v)&:=e^{\frac{1}{2}\int_t^T \left(\sigma^2 u^2 +\sigma^2_r v^2 \right)dr } \ \E\left[\int e^{-i (u \tilde z+v s)}\pi^*_{y+Z_T-Z_t}(d\tilde z,ds )\right] \\
&=e^{\frac{\sigma^2 u^2}{2}(T-t) + v^2 \int_t^T \sigma_r^2 dr} \ 
\E[e^{ i m_\beta u - i \varepsilon (-\psi \tilde\lambda^2 u + \lambda^2 v)(y+Z_T-Z_t)
-\frac{1}{2}\varepsilon^2 \lambda^2 \tilde\lambda^2 \sigma^2 T (u+\psi v)^2
}]
\\
&=\exp\Big({ \frac{\sigma^2 u^2}{2}(T-t) + v^2 \int_t^T \sigma_r^2 dr
+ i u m_\beta
- i \varepsilon (-\psi \tilde\lambda^2 u + \lambda^2 v)y}\\
& \quad \ \ \ \ -\frac{\varepsilon^2}{2}(-\psi \tilde\lambda^2 u + \lambda^2 v)^2 \sigma^2 (T-t)
-\frac{1}{2}\varepsilon^2 \lambda^2 \tilde\lambda^2 \sigma^2 T (u+\psi v)^2 \Big)\\
& = \exp\Big({ i u m_\beta
- i \varepsilon (-\psi \tilde\lambda^2 u + \lambda^2 v)y
- \frac{1}{2} u^2 \big( \sigma^2 \tilde\lambda^2(T-\varepsilon^2 \psi^2 \tilde\lambda^2 t) - \sigma^2(T-t)
\big)
}\\
& \quad \ \ \ \
- \frac{1}{2} v^2 \big( \sigma^2 \lambda^2 (T-\varepsilon^2 \lambda^2 t) - (\Sigma_T^2 - \Sigma_t^2)
\big)
- uv \varepsilon^2 \psi \tilde\lambda^2 \lambda^2 \sigma^2 t 
\Big) \ .
\end{align*}
Thus, defining 
\begin{align*}
\boldsymbol{\sigma_t} &=\begin{bmatrix}
\sigma^2 \tilde\lambda^2(T-\varepsilon^2 \psi^2 \tilde\lambda^2 t) - \sigma^2(T-t) &
\varepsilon^2 \psi \tilde\lambda^2  \lambda^2 \sigma^2 t \\
\varepsilon^2 \psi \tilde\lambda^2 \lambda^2 \sigma^2 t & \sigma^2 \lambda^2 (T-\varepsilon^2 \lambda^2 t) - (\Sigma_T^2 - \Sigma_t^2)
\end{bmatrix} 
\end{align*}
and $K_t=\det{(\boldsymbol{\sigma_t})}$, by Sylvester's criterion, if $$K_t>0 \mbox{ and }\sigma^2 \tilde\lambda^2(T-\varepsilon^2 \psi^2 \tilde\lambda^2 t) - \sigma^2(T-t)>0,$$
we can take
$\rho_t(y,\tilde z, s)$ as the density of a Normal distribution with mean
$$\mu=
\begin{bmatrix}
-m_\beta\\0
\end{bmatrix}
+{y \varepsilon} \begin{bmatrix}
-\psi \tilde\lambda^2   \\ \lambda^2 
\end{bmatrix}
$$
and covariance matrix $\boldsymbol{\sigma_t} .$
Thus, we can write
$$\rho_t(y,z,\xi)=\frac{1}{\sqrt{2\pi \det{(\boldsymbol{\sigma_t})}} } e^{-\frac{1}{2}\tilde \mu^\top \boldsymbol{\sigma_t}^{-1}\tilde \mu} \ , $$
where 
\begin{align*}
   \tilde\mu = 
   \begin{bmatrix}
   -m_\beta\\0
   \end{bmatrix}
   +y \mu_1  - \begin{bmatrix}
   \tilde z \\ s
   \end{bmatrix}, \quad \mu_1 = \varepsilon \begin{bmatrix}
-\psi \tilde\lambda^2   \\ \lambda^2 
\end{bmatrix} \ .
\end{align*}
Since 
\begin{align*}
    \pa_y \ln \rho_t(y,\tilde z, s) & = - \mu_1^\top \boldsymbol{\sigma_t}^{-1}\tilde \mu \ , \\
    \pa_{\tilde z} \ln \rho_t(y,\tilde z, s) & =  e_1^\top \boldsymbol{\sigma_t}^{-1}\tilde \mu \quad \text{ where } \quad  e_1^T=\begin{bmatrix}
1   & 0  
\end{bmatrix}\ ,
\end{align*}
we now can identify $A$ by
\begin{align*} 
A_t (y,\tilde z, s) = \sigma^2 ( \pa_y \ln \rho_t(y,\tilde z, s) + \pa_{\tilde z} \ln \rho_t(y,\tilde z, s) )
=
\sigma^2( e_1- \mu_1)^\top
\boldsymbol{\sigma_t}^{-1}\tilde \mu \ .
\end{align*}
Writing explicitly, we have
$$A_t(y,\tilde z, s) = A_0(t) y + A_1(t) (\tilde z + m_\beta) + A_2(t) s, $$ 
where
\begin{align*}
    A_0(t) & =-\frac{\sigma^2}{\det{(\boldsymbol{\sigma_t})}} \Big(\varepsilon\psi\tilde \lambda^2(\sigma^2 \lambda^2T-\Sigma_T^2+\Sigma_t^2)(1+\varepsilon\psi \tilde \lambda^2)+\sigma^2\varepsilon^2\lambda^4((\tilde \lambda^2-1) T+t)\Big)\\
    A_1(t) &=-\frac{\sigma^2}{\det{(\boldsymbol{\sigma_t})}} \Big((1+ \varepsilon\psi \tilde \lambda^2)(\sigma^2\lambda^2(T- \varepsilon^2\lambda^2 t)-\Sigma^2_T+\Sigma^2_t)+ \varepsilon^2\psi\tilde \lambda^2\lambda^4\sigma^2t\Big)\\
    A_2(t) &=\frac{\sigma^2}{\det{(\boldsymbol{\sigma_t})}} \Big((1+ \varepsilon\psi \tilde \lambda^2)(\varepsilon^2\psi\lambda^2\tilde\lambda^2\sigma^2t)+\varepsilon\lambda^2\sigma^2(\tilde\lambda^2(T-\varepsilon^2\psi^2\tilde \lambda^2t)-T+t)\Big)
\end{align*}
\end{proof}

\begin{rmk}
i) Let us now take $\psi=0$ and $\sigma_t=0$ so that the example in Lemma \ref{example:gaussian} is the example in Lemma \ref{example:Kyle-Back} with static information. 
In this 1-dimensional static information setting $\lambda$ is the same as the so-called Kyle's lambda and the pricing rule is
$$ H(t,Y_t) =\lambda Y_t \ .$$ 
The optimal strategy \eqref{eq:optimalStrat:multiD_KB} for the informed trader becomes
$$A_t(Y_t,S_0)=\frac{\frac{S_0}{\lambda}-Y_t}{T-t}$$
whereas the optimal strategy \eqref{eq:psia} is
$$
    A_t( Y_t, \tilde Z_t, S_0) 
= \frac{\frac{S_0}{\lambda}-Y_t}{T-t}  - \frac{\sigma^2}{\tilde \sigma^2_t } \left(\tilde Z_t  + m_\beta\right) \ .
$$
The strategy \eqref{eq:optimalStrat:multiD_KB} is the classical Brownian Bridge construction of Kyle-Back models as in \cite{kyle, back1992, back1993, cc2007}, whereas \eqref{eq:psia} is different from this Brownian bridge.
This novel strategy requires the informed trader to keep track of $\tilde Z_t$. But, it is significantly more tractable than the Brownian bridge. Unlike the Brownian bridge where the joint conditional law of $(\tilde Z_t,S)$ is path-dependent in $Y$ (through the the conditional expectation of $Z_t$ which is a stochastic integral driven by $Y$), for the strategy we constructed, this conditional distribution is a function of $Y_t$. As such, the filtering is Markovian. This point is crucial to use optimal transport for $\psi>0$ since the conditional law of $(\tilde Z_T, S)$ depends only on $Y_T$ and the transport maps between $(\tilde Z_T,S)$ and $Y_T$ can be easily computed. 

ii) The lack of uniqueness of strategies is a consequence of the lack of wellposedness of \eqref{eq:bspde2}. Indeed, the expressions of $A$ in i) corresponds to two distinct solutions of \eqref{eq:bspde2}. 
\end{rmk}

\subsection{A sobering counter-example}\label{ss.counter}
Although our methodology leads to the existence of equilibrium in all known continuous time Kyle model with risk-neutral agents, we show in this subsection that the BSPDE \eqref{eq:bspde2} cannot admit solutions if $I$ is not infinitely smooth in $\tilde z$, which is not expected for many choices of $V$.  Afterall, transport maps are generically not smooth.

\begin{proposition}\label{propsmooth}
Assume that the BSPDE \eqref{eq:bspde2} admits a solution and the law $\mu_0$ of $(-\beta, S_0)$ admits a smooth density. Then, $\tilde z\mapsto I(\tilde z,s)$ is $C^{\infty}$.
\end{proposition}
\begin{proof}For simplicity, we assume $n=1$ and $\sigma_t=0$ for all $t\in [0,T]$. 
For all $t\in[0,T)$, denote by $\gamma_t(\tilde z)$ the Gaussian density of $\tilde Z_T-\tilde Z_t$ and $\rho_t$ a solution to \eqref{eq:bspde2}. Denoting $\widehat{ \gamma_t*_{\tilde z} \rho_t}$ the Fourier transform of $\gamma_t *_{\tilde z} \rho_t$, where $*_{\tilde z}$ denotes the convolution in $\tilde z$, 
we have, by a direct computation,  
\begin{align*}
    \widehat{ (\gamma_t*_{\tilde z} \rho_t)}(u,v)
= & \ \E\left[ \int e^{-i (u \tilde z+v s)} \pi^*_{Y_T}(d \tilde z,d\xi) 
      \ | \ \Fc_t^Y \right] \ ,
\end{align*}
where $Y$ is a fixed $\sigma$-Brownian motion.

Inverting the Fourier transform, one finds the very simple expression
$$
    (\gamma_t *_{\tilde z} \rho_t)({\tilde z},s)d{\tilde z} ds 
=  \ \E\left[ \pi^*_{Y_T}(d \tilde z,d\xi) 
     \ | \ \Fc_t^Y \right] \ ,
$$
or equivalently
\begin{align}\label{eq:convrho}
  & (\gamma_{t} *_{\tilde z} \rho_t)({\tilde z},s)d{\tilde z} ds \nonumber\\
= & \ \int \pi^*_{y}(d { \tilde z},ds) 
      e^{-\frac{(y-Y_t)^2}{2\sigma^2(T-t)} }
      \frac{dy}{\sqrt{2\pi \sigma^2(T-t)}} \\
= & \ \int \pi^*(dy, d\tilde z,ds) 
      \sqrt{\frac{1}{1-t/T}} \   \exp{\Big( \frac{y^2}{2\sigma^2 T} -\frac{(y-Y_t)^2}{2\sigma^2(T-t)} \Big) } \ ,\notag
\end{align}
where in the last step we used the fact that $\pi^*(dy,d\tilde z,ds)=\frac{1}{\sqrt{2\pi \sigma^2 T}}e^{-\frac{y^2}{2\sigma^2 T}}dy \pi^*_y(d{\tilde z},ds)$ is an optimal coupling between the law $\nu_T$ of $Y_T$ and the law $ \mu_T$ of $(\tilde Z_T,S_T)$ . 

By assumption on $\mu_0$, $\mu_T$ admits a smooth density that we denote $f_\mu(\tilde s,\tilde z)$. Let us write the following exact formula for equation \eqref{eq:convrho} 
\begin{align*}
  & (\gamma_{t} *_{\tilde z} \rho_t)({\tilde z},s)d{\tilde z} ds \\
= & \ \int  
      \delta_{I(\tilde z , s)} (dy)f_\mu(\tilde z ,s)
      \sqrt{\frac{1}{1-t/T}} \   \exp{\Big( \frac{y^2}{2\sigma^2 T} -\frac{(y-Y_t)^2}{2\sigma^2(T-t)} \Big) }d{\tilde z} ds \\
= & \ 
      f_\mu(\tilde z ,s)
      \sqrt{\frac{1}{1-t/T}} \   \exp{\Big( \frac{I^2(\tilde z , s)}{2\sigma^2 T} -\frac{(I(\tilde z , s)-Y_t)^2}{2\sigma^2(T-t)} \Big) }d{\tilde z} ds \ .
\end{align*}
Note that $\tilde z\mapsto  (\gamma_{t} *_{\tilde z} \rho_t)({\tilde z},s)$ is a convolution with a Gaussian density. Thus, we obtain that $\tilde z\mapsto \exp{\Big( \frac{I^2(\tilde z , s)}{2\sigma^2 T} -\frac{(I(\tilde z , s)-Y_t)^2}{2\sigma^2(T-t)} \Big) }$ has to be smooth for all $t\in(0,T)$ which imposes that 
$\tilde z\mapsto I(\tilde z,s)$ is smooth. 

\end{proof}

\section{Conclusion}
In this paper, we described a general methodology for solving the Kyle model. Although our methodology does not systematically solve the problem, it encompasses all the examples in literature. 

As described above, standard constructions of equilibrium in Kyle model consist in solving two problems successively. First, a problem at time $T$ which is the optimal transport problem \eqref{eq:OT}. This problem leads to the construction of the pricing rule of the market maker \eqref{eq:defH} and an optimality condition for the strategy of the informed trader \eqref{eq:target}. Then, one needs to solve a second problem on $[0,T)$ which is the construction of a type of bridge that can be written in all known cases as the BSPDE \eqref{eq:bspde2}. The solution to this BSPDE leads to an inconspicuous strategy for the informed trader and to an equilibrium.

However, the Proposition \ref{propsmooth} shows that this methodology which is widely used in the literature imposes a strong regularity condition on the transport map. In fact, it is our understanding that the condition at hand is not a problem of regularity but a problem of causality. Indeed, the optimality condition \eqref{eq:target} requires that the informed trader generates the optimal coupling $\pi^*$ of $((\tilde Z_T,S_T),Y_T)$ at the final time using trading strategies of type 
\eqref{eq:dynY}. Thus, the coupling between the processes $(\tilde Z_t,S_t)_{t\in[0,T]}$ and the Brownian motion $(Y_t)_{t\in[0,T]}$ has to be causal in the sense of \cite{lassalle2013causal,acciaio2020causal}. This causality requirement is ignored in the classical constructions. The transport problem \eqref{eq:OT} is optimized among all couplings of the terminal values of $(\tilde Z_t,S_t)_{t\in[0,T]}$ and the Brownian motion $(Y_t)_{t\in[0,T]}$ -- instead of restricting to the causal ones. 

\bibliographystyle{halpha}
\bibliography{ref.bib}

\end{document}